\documentclass[12pt,reqno]{amsart}

\usepackage{graphicx}
\graphicspath{ {./images/} }
\usepackage{bm}

\usepackage{color}

\usepackage{amssymb}

\usepackage{amsfonts}

\usepackage{amsmath}

\usepackage[all]{xy}

\usepackage[colorlinks=true, allcolors=blue]{hyperref}

\newtheorem{theorem}{Theorem}[section]

\newtheorem{corollary}[theorem]{Corollary}

\newtheorem{lemma}[theorem]{Lemma}

\newtheorem{proposition}[theorem]{Proposition}

\newtheorem{Definition}[theorem]{Definition}

\newtheorem{Example}[theorem]{Example}

\newtheorem{Remark}[theorem]{Remark}

\newenvironment{remark}{\begin{Remark}\begin{em}}{\end{em}\end{Remark}}

\newenvironment{definition}{\begin{Definition}\begin{em}}{\end{em}\end{Definition}}

\DeclareMathOperator{\tr}{\mathrm{tr}}
\DeclareMathOperator{\supp}{\mathrm{supp}}

\address{Sejong Kim \\ Department of Mathematics, Chungbuk National University, Cheongju 28644, Korea}
\email{skim@chungbuk.ac.kr}

\address{Vatsalkumar N. Mer \\ Institute for Industrial and Applied Mathematics, Chungbuk National University, Cheongju 28644, Korea}
\email{vnm232657@gmail.com}

\address{Mikl\'os P\'alfia \\ Department of Mathematics, Corvinus University of Budapest, 1093 Budapest Fővám tér 8, Hungary}
\address{Bolyai Institute, Interdisciplinary Excellence Centre, University of Szeged, H-6720 Szeged, Hungary}
\email{miklos.palfia@uni-corvinus.hu}

\begin{document}

\title[On the Wasserstein barycenter of positive definite operators]{On the Wasserstein barycenter of positive definite operators}

\author{Sejong Kim, Vatsalkumar N. Mer and Mikl\'os P\'alfia}

\date{\today}

\begin{abstract}
We extend the Bures-Wasserstein mean of positive definite matrices to the case of positive definite operators on a Hilbert space. This is done through its defining stationary point operator equation, coming from the gradient of the sum of squared Bures-Wasserstein distances of centered Gaussians represented by positive definite matrices. This gradient is shown to have a Fr\'echet derivative which induces a bounded linear operator on the space of Hilbert-Schmidt operators with strictly positive real spectrum. This allows us to conclude the existence and uniqueness of this mean by exhibiting the spectral permanence of this operator when extended to general bounded linear operators and also enables the study of its generated ODE semigroups, which enjoy exponential contraction in a Banach-Finsler metric obtained through the construction of equivalent renormings. Using this exponential contractivity of the flow, we prove a `Nodice'-type of theorem and its stochastic variant, a Sturm-type of strong law of large numbers for probability measures with bounded support. We also verify fundamental properties and establish various operator inequalities satisfied by the Wasserstein mean.

\vspace{5mm}


\end{abstract}

\maketitle

\section{Introduction}
The theory of geometric means of positive definite matrices and operators has substantially studied over the last two decades, motivated by applications in matrix analysis, operator theory, information geometry, statistics, signal processing, and quantum information.
While the theory of two-variable operator means was established through the theoretic framework of Kubo and Ando \cite{KA}, the extension to several variables has been considered more challenging due to the lack of a canonical  multivariable analogue satisfying the fundamental properties expected by a geometric mean.
The first attempt of positive definite matrices is a symmetrization procedure introduced by Ando, Li, and Mathias \cite{ALM}, who suggested axiomatic properties for  multivariable geometric mean including joint homogeneity, monotonicity, congruence invariance, self-duality, joint concavity, and the arithmetic-geometric-harmonic mean inequalities.

A decisive breakthrough came from the Riemannian geometric viewpoint. Building on ideas originating in the work of Cartan and Karcher on Riemannian centers of mass, Moakher \cite{Mo} and subsequently Bhatia and Holbrook \cite{BH} proposed defining the  multivariable geometric mean as the unique minimizer of the weighted sum of squared Riemannian trace distances.
This least squares mean of positive definite matrices $A_{1}, \dots, A_{m}$ is characterized as the unique positive definite solution $X$ of the nonlinear matrix equation
\begin{equation} \label{E:Karcher}
\sum_{i=1}^{m} w_i \log \bigl( X^{-1/2} A_i X^{-1/2} \bigr) = 0,
\end{equation}
where $\omega = (w_{1}, \dots, w_{m})$ is a positive probability vector.
Equation \eqref{E:Karcher} is known as the Karcher equation so we commonly call the barycenter the \emph{Karcher mean} and write it as $\Lambda(\omega; A_1, \ldots, A_m)$.
The Karcher mean quickly emerged as a distinguished  multivariable geometric mean because of its close connection with the intrinsic geometry of the positive cone.
A central question was whether it satisfies the collection of Ando-Li-Mathias properties.
The monotonicity problem remained open for several years until Lawson and Lim \cite{LL11}, and independently Bhatia and Karandikar \cite{BK}, established it using fundamentally different approaches.
These developments confirmed the Karcher mean as a natural multivariable extension of the geometric mean and stimulated extensive further investigation.

One of the remarkable approaches to the Karcher mean was achieved by Lim and Pálfia \cite{LP12} through the introduction of  multivariable power means.
For $t \in (0,1]$, the power mean $P_t(\omega; A_1, \ldots, A_m)$ is defined as the unique positive definite solution $X$ of the equation
\begin{displaymath}
X = \sum_{i=1}^{m} w_i (X \#_t A_i),
\end{displaymath}
where $A \#_t B = A^{1/2} (A^{-1/2} B A^{-1/2})^{t} A^{1/2}$ denotes the weighted geometric mean of $A$ and $B$.
They proved that
\begin{displaymath}
P_t(\omega; A_1, \ldots, A_m) \longrightarrow \Lambda(\omega; A_1, \ldots, A_m) \quad \textrm{as} \quad t \to 0.
\end{displaymath}
This approximation viewpoint not only supplied new proofs of fundamental properties but also became the key mechanism for extending the theory beyond finite dimensions.

The infinite-dimensional theory was developed by Lawson and Lim \cite{LL14}, who showed that the Karcher equation admits a unique positive definite solution for bounded positive invertible operators on a Hilbert space. Since the affine-invariant Riemannian structure is no longer available in this setting, the extension relied on the approximation by power means together with the geometry induced by the Thompson metric. Their work established that the resulting operator mean retains essentially all of the characteristic properties known in finite dimensions, including joint homogeneity, monotonicity, congruence invariance, self-duality, and continuity. In parallel, they clarified in \cite{LL} the relationship between the Karcher mean and the least squares mean on the manifold of positive Hilbert–Schmidt operators, showing that the former can be viewed as a unique monotone strongly continuous extension of the latter.

Subsequent research shifted from finite tuples of operators to probability measures on the positive cone. This measure-theoretic perspective is motivated both by applications and by analogies with barycenters in metric geometry. P\'{a}lfia \cite{palfia2} developed a general framework in which operator means are characterized as unique solutions of generalized Karcher equations obtained by replacing the logarithm with an arbitrary normalized operator monotone function. This theory simultaneously extends Kubo–Ando means, matrix power means, and the Karcher mean, and reveals a unifying fixed-point and contraction-theoretic structure behind multivariable operator means.

New Riemannian distance, different from the affine-invariant trace metric on the cone of positive definite matrices, has been introduced:
\begin{displaymath}
d_{W}(A, B) = \left[ \mathrm{tr} (A + B) - 2 \mathrm{tr} (A^{1/2} B A^{1/2})^{1/2} \right]^{1/2}.
\end{displaymath}
In optimal transport theory  \cite{V} and statistics \cite{PZ}, it is exactly equivalent to the $L_{2}$-Wasserstein metric of two zero-mean Gaussian distributions whose covariance matrices are $A, B$.
Under the special case that $A, B$ commute, this metric reduces to the classical Hellinger distance between their eigenvalue distributions.
Furthermore, this metric of density matrices in quantum information coincides with the Bures distance so we call $d_{W}$ the Bures-Wasserstein metric.

Given positive definite matrices $A_{1}, \dots, A_{m}$, the least squares mean with respect to Bures-Wasserstein distance is given by
\begin{displaymath}
\underset{X > 0}{\arg \min} \sum_{i=1}^{m} w_{i} d_{W}^{2}(X, A_{i}).
\end{displaymath}
We call the unique minimizer the \emph{Wasserstein mean} and write as $\Omega(\omega; A_{1}, \dots, A_{m})$.
Although the geometry of the cone $\mathbb{P}_{n}$ of positive definite matrices equipped with the Bures-Wasserstein metric exhibits non-negative curvature implying that it is not a Hadamard space, this minimization problem is well-known to possess a unique positive definite solution.
For existence and uniqueness, we refer \cite{AC} about a method of non-smooth analysis, convex duality and optimal transport theory and \cite{BJL} about a method of matrix analysis.
A fundamental algebraic characterization of this unique barycenter is
coincides with a unique positive definite solution to the following non-linear matrix equation
\begin{equation} \label{E:gradF}
X = \sum_{i=1}^{m} w_{i} (X^{1/2} A_{i} X^{1/2})^{1/2},
\end{equation}
obtained by vanishing the gradient of objective function. This mean with two variables $A, B$ is the geodesic line connecting two matrices under the Bures-Wasserstein metric and has an explicit algebraic formula
\begin{displaymath}
\gamma(t) = (1-t)^{2} A + t^{2} B + t(1-t) \left[ A (A^{-1} \#_{1/2} B) + (A^{-1} \#_{1/2} B) A \right].
\end{displaymath}
Furthermore, many properties of the Wasserstein mean of positive definite matrices have been studied: a fixed point algorithm to the Wasserstein mean \cite{ABCM}, a log-majorization relationship with the Karcher mean and log-Euclidean mean \cite{BJL-1} and an extended Lie-Trotter-Kato formula for the Wasserstein mean \cite{HK19}. More recently, related developments include dimension-free convergence results for gradient descent on the Bures--Wasserstein manifold \cite{ACGS}, entropic regularization of Wasserstein distances in infinite-dimensional Gaussian settings \cite{M-2}, and large-sample theory for Bures--Wasserstein barycenters \cite{SP}.

The main purpose of this paper is to establish the Wasserstein mean of positive definite operators over a Hilbert space.
We first show the existence and uniqueness of a minimizer for the least squares problem in the space of extended trace-class positive operators, then move over to the Hilbert-Schmidt and finally to the general case, the cone $\mathbb{P}$ of positive definite operators. The first-order optimality condition manifests as a non-linear fixed point equation, traditionally known as the Wasserstein equation \eqref{E:gradF}.
While the local unique solvability of this Wasserstein equation \eqref{E:gradF} can be established near the identity operator by using the implicit function theorem, extending this uniqueness globally across the entire Loewner interval $[\alpha I, \beta I]$ requires a sophisticated topological and dynamical machinery.
To this end, we introduce a continuation (homotopy) method via the deformation path $\displaystyle \phi(t,X) = \sum_{i=1}^{m} w_{i} (A_{i}^{t} \#_{1/2} X^{-1}) - I$, by showing that the Fr\'{e}chet derivative remains uniformly bounded below and invertible along the trajectory.
Furthermore, by considering ODEs of the the underlying vector field $\displaystyle \Phi(X) = \sum_{i=1}^{m} w_{i} (I - (A_{i} \#_{1/2} X^{-1}))$ and leveraging the theory of contraction semigroups on weighted Banach-Finsler manifolds, we construct a convergent forward Euler step method.
This dynamic approach guarantees that the recursive iteration
\begin{displaymath}
X_{k+1} = X_{k} - \frac{1}{k+1} \Phi(X_{k}) + O \left( \frac{1}{(k+1)^2} \right)
\end{displaymath}
which is similar to Euclidean gradient descent in the case of matrices, converges to the unique multi-variable Wasserstein mean in a Banach-Finsler metric $d_{\epsilon}$ that we construct by introducing a continuous section of Banach renormings equivalent to the original Banach space norm given by the operator norm. The key point here is that the spectrum $\sigma(-D\Phi(X)(\cdot))$ is contained in the negative real half-line, thus the flows of the ODE
\begin{align*}
    u_i'(t) &= - \Phi(u_i(t)), \\
    u_i(0) &\in\mathbb{P},
\end{align*}
exhibit exponential contraction property
\begin{equation*}
d_\epsilon(u_1(t), u_2(t)) \leq e^{(\sup\{\Re\sigma(-D\Phi(X)):X\in S\}+\epsilon) t} d_\epsilon(u_1(0), u_2(0))
\end{equation*}
where $S$ is a large enough order interval in $\mathbb{P}$ containing each $A_i$.
Further we establish convergence of deterministic and stochastic versions of the Euler step method with splitting. This can be thought of as a discrete-time Trotter-Kato formula for this nonlinear semigroup in the ambient space of bounded linear operators of which the deterministic case is known as the 'Nodice theorem', while the stochastic version can be thought of as a version of a nonlinear strong law of large numbers in the sense of Sturm \cite{sturm}, with further existing variants for Karcher means of positive operators \cite{lekapalfia, LP,  PalfiaStrong}.

Finally, we analyze the essential algebraic properties of this multi-variable operator mean including homogeneity, permutation invariance, and unitary congruence invariance and establish bounds of the Wasserstein mean with respect to the Loewner order and operator norm.
Moreover, we investigate the relationship between the Wasserstein mean and quasi-arithmetic mean, and Lim-P\'{a}lfia's power mean, and also study the connection between the Wasserstein mean and its dual.



\section{Bures-Wasserstein distance of trace-class positive definite operators}

Let $B(\mathcal{H})$ be the Banach space of all bounded linear operators on a Hilbert space $\mathcal{H}$ with inner product $\langle \cdot, \cdot \rangle$, and let $S(\mathcal{H}) \subset B(\mathcal{H})$ be the closed subspace of all self-adjoint linear operators. We say that $A \in S(\mathcal{H})$ is positive semi-definite (positive definite) if $\langle x, Ax \rangle \geq (>) 0$ for all (non-zero, respectively) vector $x \in \mathcal{H}$. For $X, Y \in S(\mathcal{H})$ we define as $X \leq (<) Y$ if $Y - X$ is positive semi-definite (positive definite, respectively). It is a partial order, known as the Loewner order.
We denote as $\mathbb{P} \subset S(\mathcal{H})$ the convex cone of all invertible positive definite operators. Let $U(\mathcal{H})$ denote the set of unitary operators on $\mathcal{H}$.

%
%
%

Let $L^2(\mathcal{H})$ denote the space of Hilbert--Schmidt operators on a Hilbert space $\mathcal{H}$, and let $S^2(\mathcal{H})$ denote the Hilbert--Schmidt class of self-adjoint operators on $\mathcal{H}$.
 We define the set of extended (unitized) Hilbert-Schmidt operators, as previously introduced in \cite{L, LL}, as follows:
\begin{displaymath}
L^{2}_{ex} (\mathcal{H}) := \{ A + \gamma I: A \in L^{2} (\mathcal{H}), \gamma \in \mathbb{R} \},
\end{displaymath}
where $I$ is the identity operator. The space $L^2_{ex}(\mathcal{H})$ forms a Hilbert space with the following inner product:
\begin{displaymath}
\langle A + \lambda I, B +\nu I \rangle_{ex2} = \langle A, B \rangle_2 + \lambda \nu = \tr(AB^*) + \lambda \nu
\end{displaymath}
where $A, B \in L^2(\mathcal{H})$ and $\lambda, \nu \in \mathbb{R}$. So the induced norm is given by
\begin{displaymath}
\Vert A + \lambda I \Vert_{ex2} = \langle A + \lambda I, A + \lambda I \rangle_{ex2}^{1/2}.
\end{displaymath}
We denote the symmetric part of $L^2_{ex}(\mathcal{H})$ by
\begin{displaymath}
SL^{2}_{ex} (\mathcal{H}) = \{ A + \gamma I:  A = A^{*}, A \in L^{2} (\mathcal{H}), \gamma \in \mathbb{R} \} = L^{2}_{ex} (\mathcal{H}) \cap S(\mathcal{H}),
\end{displaymath}
which forms a real Hilbert space with the same inner product. The positive part of this space is denoted by $\Sigma_{2} = \mathbb{P} \cap SL^{2}_{ex} (\mathcal{H})$.

Denote by $L^1(\mathcal{H})$ the space of trace-class operators on a Hilbert space $\mathcal{H}$, and by $S^1(\mathcal{H})$ the corresponding class of self-adjoint trace-class operators. We define the set of extended (unitized) trace-class operators, as previously introduced in \cite{M}, as follows:
\begin{displaymath}
L^{1}_{ex} (\mathcal{H}) := \{ A + \gamma I: A \in L^{1} (\mathcal{H}), \gamma \in \mathbb{R} \}.
\end{displaymath}
We can define the extended trace norm and trace on $L^1_{ex}(\mathcal{H})$. For any $A + \gamma I \in L^1_{ex}(\mathcal{H})$, the extended trace norm is defined by
\begin{displaymath}
\| A + \gamma I \|_{ex1} :=  \| A \|_1 + | \gamma |,
\end{displaymath}
and the extended trace, denoted by $\mathrm{tr}_{ex}$, is defined by
\begin{displaymath}
\text{tr}_{ex}(A + \gamma I) := \text{tr}(A) + \gamma.
\end{displaymath}
Finally, we define the positive part of $SL^1_{ex}(\mathcal{H}) = L^1_{ex}(\mathcal{H}) \cap S(\mathcal{H})$ as
\begin{displaymath}
\Sigma_1 := \mathbb{P} \cap SL^{1}_{ex} (\mathcal{H}) = \{ A + \gamma I > 0 :  A = A^{*}, A \in L^{1} (\mathcal{H}), \gamma \in \mathbb{R} \}.
\end{displaymath}

The Wasserstein distance (or Bures-Wasserstein metric) of positive trace-class operators is defined as
\begin{equation*} \label{metric}
d_W(A,B) : =  \left[ \mathrm{tr} ( A + B ) - 2 \mathrm{tr} (A^{1/2} B A^{1/2})^{1/2} \right]^{1/2}.
\end{equation*}
The following provides the extremal characterization of Wasserstein distance of positive trace-class operators.
\begin{theorem}[\cite{G}, \cite{MPZ}]
If $d_W(A,B)$ is defined as in above, then
\begin{align*}
d_W(A,B) = \min_{U \in U(\mathcal{H})} \| A^{1/2} - B^{1/2}U \|_2.
\end{align*}
\end{theorem}

\begin{definition}
The Procrustes distance between $A + \gamma I, B + \gamma I \in \Sigma_1$ is defined as
\begin{displaymath}
d_{\mathrm{pro}}(A + \gamma I, B + \gamma I) : = \min_{(I+U) \in U(\mathcal{H}) \cap L^{1}_{ex} (\mathcal{H})} \| (A + \gamma I)^{1/2} - (B + \gamma I)^{1/2} (I + U) \|_{ex2}.
\end{displaymath}
\end{definition}

The following shows the relationship between the procrustes distance and Wasserstein distance on $\Sigma_{1}$.
\begin{theorem}[\cite{M-1}] \label{T:Wass-form}
Let $A + \gamma I, B + \gamma I \in \Sigma_{1}$. Then
\begin{displaymath}
\begin{split}
& d^2_{\mathrm{pro}}(A + \gamma I, B + \gamma I) \\
&= 4 \mathrm{tr}_{ex} \left[ (A + \gamma I) + (B + \gamma I) - 2 ((A + \gamma I)^{1/2} (B + \gamma I) (A + \gamma I)^{1/2})^{1/2} \right],
\end{split}
\end{displaymath}
and
\begin{align*}
\lim_{\gamma \rightarrow 0} d_{\mathrm{pro}}(A + \gamma I, B + \gamma I) = 2 d_W(A, B).
\end{align*}
\end{theorem}


\section{The Wasserstein Mean of  positive definite operators}

We begin by defining the Wasserstein mean for extended (unitized) Hilbert–Schmidt positive definite operators, and then extend the definition to positive definite operators. We consider the minimization problem as finding the least squares mean of $A_1, A_2 , \ldots, A_m \in \Sigma_1 $ as follows:
\begin{equation} \label{Wass-mean-1}
\underset{X \in \Sigma_1}{\arg \min} \sum_{i=1}^{m} w_i d_{\mathrm{pro}}^2(X, A_i),
\end{equation}
where $\omega = (w_1, \ldots, w_m)$ represents a probability vector: $w_i > 0$ for all $i$ and $\sum_{i=1}^{m} w_i = 1$. We show the existence and uniqueness of a minimizer for the objective function $F(X) = \sum_{i=1}^{m} w_i d_{\mathrm{pro}}^2(X, A_i)$. This unique solution is called the Wasserstein barycenter of $A_1, A_2, \ldots, A_m \in \Sigma_1$ and denoted as $\Omega(\omega; A_1, A_2, \ldots, A_m)$.

We denote as $\Delta_{m}$ the set of all probability vectors in $\mathbb{R}^{m}$.
\begin{theorem}
Let $(A_1, A_2, \ldots, A_m) \in \Sigma_1^m$ and $\omega = (w_{1}, w_{2}, \dots, w_{m}) \in \Delta_{m}$. Then the minimization problem \eqref{Wass-mean-1} has a unique solution, which satisfies the operator equation \eqref{gradF}.
\end{theorem}

From Theorem \ref{T:Wass-form} we have
\begin{equation*}\label{F}
F(X) = \sum_{i=1}^{m} 4 w_i \mathrm{tr}_{ex}(A_i) + \sum_{i=1}^{m} 4 w_i \mathrm{tr}_{ex}(X -2 (A_i^{1/2} XA_i^{1/2})^{1/2}).
\end{equation*}
We show that $F(X)$ is strictly convex. It is enough to show that $H(X) = \mathrm{tr}_{ex}(X^{1/2})$ is strictly concave.

\begin{theorem}
The function $H(X) = \mathrm{tr}_{ex}(X^{1/2})$ is strictly concave on $\Sigma_1$, i.e. for $\alpha, \beta > 0$ with $\alpha + \beta = 1$ and $A, B \in \Sigma_1$
\begin{displaymath}
H(\alpha A + \beta B) \geq \alpha H(A) + \beta H(B),
\end{displaymath}
where the equality holds if and only if $A = B$.
\end{theorem}

\begin{proof}
We know that the square-root function is operator concave, i.e.,
\begin{displaymath}
(\alpha A + \beta B)^{1/2} \geq \alpha A^{1/2} + \beta B^{1/2}.
\end{displaymath}
So taking the extended trace on both sides yields $H(\alpha A + \beta B) \geq \alpha H(A) + \beta H(B)$.

For $A \neq B$, assume that $H(\alpha A + \beta B) = \alpha H(A) + \beta H(B).$
Since $(\alpha A + \beta B)^{1/2} \geq \alpha A^{1/2} + \beta B^{1/2}$, it follows that $(\alpha A + \beta B)^{1/2} = \alpha A^{1/2} + \beta B^{1/2}$.
Squaring both sides and utilizing the identity $\alpha - \alpha^2 = \beta - \beta^2 = \alpha \beta$, we obtain
\begin{displaymath}
\alpha \beta (A^{1/2} - B^{1/2})^{2} = 0.
\end{displaymath}
This implies $A = B$, which contradicts our assumption. Thus, $H(X)$ is a strictly concave function on $\Sigma_1$.
\end{proof}

As mentioned, we want to show that $F(X)$ has a unique minimizer. We have already established that $F(X)$ is strictly convex. The derivative $DF(X)$, as calculated  in \cite{BJL}, is given by:
\begin{displaymath}
D F(X)(Y) = \sum_{i=1}^{m} 4 w_i \mathrm{tr}_{ex} (Y - (A_i \# X^{-1}) Y)
\end{displaymath}
From $D F(X) = 0$, we get
\begin{equation*}
\sum_{i=1}^{m} w_i A_i \# X^{-1} = I,
\end{equation*}
equivalently by congruence transformation,
\begin{equation} \label{gradF}
\sum_{i=1}^{m} w_{i} (X^{1/2} A_{i} X^{1/2})^{1/2} = X.
\end{equation}

Let $\mathbb{A} = (A_{1}, \dots, A_{m}) \in \mathbb{P}^{m}$, and let $\omega = (w_{1}, \dots, w_{m}) \in \Delta_{m}$. We denote as $\mathcal{W}(\omega; \mathbb{A})$ the set of all fixed points in $\mathbb{P}$ of \eqref{gradF}.
Note that
\begin{displaymath}
\begin{split}
\mathcal{W}(\omega; \mathbb{A}) & = \left\{ S \in \mathbb{P} : S = \sum_{i=1}^{m} w_{i} ( S^{1/2} A_{i} S^{1/2} )^{1/2} \right\} \\
& = \left\{ S \in \mathbb{P} : \sum_{i=1}^{m} w_{i} (A_{i} \# S^{-1}) = I \right\}
\end{split}
\end{displaymath}

\begin{lemma} \label{ineq}
Let $\mathbb{A} = (A_1, A_2, \ldots, A_m) \in \mathbb{P}^{m}$ and $\omega = (w_{1}, w_{2}, \dots, w_{m}) \in \Delta_{m}$ with $\alpha I \leq A_i \leq \beta I $ for all $i$ and some $\alpha, \beta > 0$. Then $X \in \mathcal{W}(\omega; \mathbb{A})$ satisfies
\begin{displaymath}
\alpha I \leq X \leq \beta I.
\end{displaymath}
\end{lemma}

\begin{proof}
Let $X \in \mathcal{W}(\omega; \mathbb{A})$ . Since $\alpha I \leq A_{i} \leq \beta I$ for all $i$,
\begin{displaymath}
\sqrt{\alpha} X^{1/2} \leq (X^{1/2} A_{i} X^{1/2})^{1/2} \leq \sqrt{\beta} X^{1/2}.
\end{displaymath}
We get
\begin{displaymath}
\sqrt{\alpha} X^{1/2} \leq \sum_{i=1}^{m} w_{i} (X^{1/2} A_{i} X^{1/2})^{1/2} = X \leq \sqrt{\beta} X^{1/2}.
\end{displaymath}
Solving the above for $X$, we obtain the desired inequalities.
\end{proof}

\begin{lemma} \label{homo_set} $($Homogeneity$)$
$\mathcal{W}(\omega; \alpha\mathbb{A})= \alpha \mathcal{W}(\omega; \mathbb{A})$ for any $\alpha > 0$.
\end{lemma}

In the following, the Loewner order interval is defined by
$ [\alpha I, \beta I] : = \{ X \in \mathbb{P}: \alpha I \leq X \leq \beta I \} $
for $0< \alpha< \beta $.

%

\begin{theorem} [Implicit Function Theorem \cite{L-b}, Theorem 5.9]
Let $X$ and $Y$ be Banach spaces, $U$ an open subset of $X \times Y$, and $(x_0, y_0)$ a point in $U$ such that $f(x_0, y_0) = 0$, where $f : U \rightarrow Y$ is continuously Fréchet differentiable with respect to $y$ at $(x_0, y_0)$, and $\frac{\partial f}{\partial y}$ is invertible at $(x_0, y_0)$.
Then there exist open sets $V \subseteq X$ containing $x_0$ and $W \subseteq Y$ containing $y_0$, and a unique continuously differentiable function $g : V \rightarrow W$ such that $f(x, g(x)) = 0$ for all $x \in V$ and $g(x_0) = y_0$.
\end{theorem}

\begin{theorem} \label{T:Wass-epsilon}
Let $\omega = (w_{1}, w_{2}, \dots, w_{m}) \in \Delta_{m}$. Then there exists $\epsilon_{\omega} > 1$ such that for any $\mathbb{A} \in [\epsilon_{\omega}^{-1} I, \epsilon_{\omega} I]^m$, the equation \eqref{gradF} has a unique solution in $[\epsilon_{\omega}^{-1} I, \epsilon_{\omega} I]$.
\end{theorem}

\begin{proof}
We define the function $H^{\omega} : \mathbb{P}^m \times \mathbb{P} \rightarrow S(\mathcal{H})$ by
\begin{displaymath}
H^{\omega} (\mathbb{A}, X) := X - \sum_{i=1}^{m} w_{i} (X^{1/2} A_{i} X^{1/2})^{1/2},
\end{displaymath}
for $\mathbb{A} = (A_1, A_2, \ldots, A_m) \in \mathbb{P}^m$ and $\omega = (w_{1}, w_{2}, \dots, w_{m}) \in \Delta_{m}$. We simply write $H^{\omega}_{\mathbb{A}}(X) := H^{\omega}(\mathbb{A}, X)$. The function $H^{\omega}: \mathbb{P}^m \times \mathbb{P} \rightarrow S(\mathcal{H})$ is continuously Fréchet differentiable.
The derivative $D_2 H^{\omega}_{\mathbb{A}}(X)$ maps from $S(\mathcal{H})$ to $S(\mathcal{H})$. Considering $\mathbb{I} = (I, \ldots, I) \in \mathbb{P}^m$, we find $H^{\omega}_{\mathbb{I}}(X) = X - X^{1/2}$ and $H^{\omega}_{\mathbb{I}}(I) = 0$.

The derivative $D_2 H^{\omega}_{\mathbb{I}}(X)(Y)$ is given by
\begin{displaymath}
D_2 H^{\omega}_{\mathbb{I}}(X)(Y) = Y - \int_0^\infty e^{-t X^{1/2}} Y e^{-t X^{1/2}} \, dt
\end{displaymath}
At $X = I$,
\begin{displaymath}
D_2 H^{\omega}_{\mathbb{I}}(I)(Y) = Y - \int_0^\infty e^{-2t} \, dt \, Y = Y - \frac{1}{2} Y = \frac{1}{2} Y.
\end{displaymath}
Thus, $D_2 H^{\omega}_{\mathbb{I}} (I) : S (\mathcal{H}) \rightarrow S (\mathcal{H})$ is an invertible map.
By Implicit Function Theorem, there exist open sets $V \subseteq \mathbb{P}^m$ containing $\mathbb{I}$ and $W \subseteq \mathbb{P}$ containing $I$, and a unique continuously differentiable function $G : V \rightarrow W$ such that $H^{\omega}_{\mathbb{B}} ( G(\mathbb{B})) = 0$ for all $\mathbb{B} \in V$ and $G(\mathbb{I}) = I$.

Choose $\epsilon >1$ such that $[\epsilon^{-1} I , \epsilon I]^m \subset V$ and $[\epsilon^{-1} I, \epsilon I] \subset W$. Then $\mathbb{A} \in [\epsilon^{-1} I, \epsilon I]^m$.
For $X \in \Gamma(\omega, \mathbb{A} )$,
\begin{displaymath}
\epsilon^{-1} I \leq X \leq \epsilon I
\end{displaymath}
by Lemma \ref{ineq}.
Therefore, $X \in [\epsilon^{-1} I, \epsilon I] \subset W$. From the previous discussion, there exists a unique continuously differentiable function $G: [\epsilon^{-1} I , \epsilon I]^m \to [\epsilon^{-1} I , \epsilon I]$ such that $H^{\omega}_{\mathbb{A}} (X) = 0$ and $G(\mathbb{A}) = X$, which means $\mathcal{W}(\omega; \mathbb{A}) = \{ X \}$.
Hence, the equation \eqref{gradF} has a unique solution in $[\epsilon_{\omega}^{-1} I, \epsilon_{\omega} I]$.
\end{proof}

\begin{corollary}\label{corol_P}
Let $\mathbb{A } =(A_1, A_2, \ldots, A_m) \in \mathbb{P}^{m}$ and $\omega = (w_{1}, w_{2}, \dots, w_{m}) \in \Delta_{m}$ with $\alpha I \leq A_j \leq \beta I$ for all $j$ and some $\alpha, \beta > 0$. If $\epsilon_{\omega}^2 \geq \beta / \alpha$, then the equation \eqref{gradF} has a unique solution.
\end{corollary}

\begin{proof}
One can easily see that
\[
I \leq \frac{1}{\alpha} A_j \leq \frac{\beta}{\alpha} I
\]
for all \( j \) and for some constants \( \alpha, \beta > 0 \).
Since \( \epsilon_{\omega}^2 \geq \frac{\beta}{\alpha} \), it follows that
\[
\frac{1}{\alpha \epsilon_{\omega}} A_j \in
\left[ \, \epsilon_{\omega}^{-1} I,\, \epsilon_{\omega} I \, \right].
\]
By Theorem~\ref{T:Wass-epsilon} and Lemma~\ref{homo_set},
this implies that \eqref{gradF} admits a unique solution.
\end{proof}


In particular, we have


\begin{theorem}
Let  $\omega = (w_{1}, w_{2}, \dots, w_{m}) \in \Delta_{m}$. Then there exists $\epsilon_{\omega} > 1$ such that for any  $(A_1, A_2, \ldots, A_m) \in (\Sigma_2 \cap [\epsilon_{\omega}^{-1} I, \epsilon_{\omega} I])^m$, the equation \eqref{gradF} has a unique solution $X$ in $\Sigma_2$.
\end{theorem}

\begin{corollary}\label{corl}
Let $\mathbb{A } =(A_1, A_2, \ldots, A_m) \in \Sigma_2^{m}$ and $\omega = (w_{1}, w_{2}, \dots, w_{m}) \in \Delta_{m}$ with $\alpha I \leq A_j \leq \beta I$ for all $j$ and some $\alpha, \beta > 0$. If $\epsilon_{\omega}^2 \geq \beta / \alpha$, then the equation \eqref{gradF} has a unique solution in $[\alpha I, \beta I ]$.
\end{corollary}

\begin{lemma}\label{l-1}
Let $A,B\in L^2_{ex}(\mathcal{H})$. Then, for any  $Y \in SL^{2}_{ex} (\mathcal{H})$,
\[
 \lambda_{\min}(A^*A)\,\lambda_{\min}(BB^*)\;
 \langle Y, Y \rangle_{ex2}\le\;
 \| A Y B\|_{ex2}^2
 \;\le\;
 \lambda_{\max}(A^*A)\,\lambda_{\max}(BB^*)\;\langle Y, Y \rangle_{ex2}.
\]
\end{lemma}

\begin{proof}
First note that
\[
\| AYB \|_{ex2}^{2}
= \tr_{ex} \big( (AYB)^* (AYB) \big)
= \tr_{ex} \big(B^* Y A^* A Y B \big)
= \tr_{ex} \big( (A^*A) \, Y \, (BB^*) \, Y \big).
\]
Let
$
\lambda_{\min}(BB^*)=\min\sigma(BB^*)$ and
$\lambda_{\max}(BB^*)=\max\sigma(BB^*)$. Then $0<\lambda_{\min}(BB^*)I\le BB^*\le \lambda_{\max}(BB^*)I$,
and
$$
\lambda_{\min}(BB^*)\,Y^{2}\ \le\ Y(BB^*)Y\ \le\ \lambda_{\max}(BB^*)\,Y^{2}.$$
We get
\[
\lambda_{\min}(BB^*) \, \tr_{ex} \big( (A^*A) Y^{2} \big)
\le \tr_{ex} \big( A \, Y \, (BB^*) \, Y A^* \big)
\le \lambda_{\max}(BB^*) \, \tr_{ex} \big( (A^*A) Y^{2} \big).
\]
Next we have
\[
\lambda_{\min}(A^*A) \, \tr_{ex}(Y^{2})
\le \tr_{ex} \big( (A^*A) Y^{2} \big) = \tr_{ex}(Y(A^*A)Y)
\le \lambda_{\max}(A^*A) \, \tr_{ex}(Y^{2}).
\]
Combining the last two inequalities, we obtain
\[
\lambda_{\min}(A^*A)\,\lambda_{\min}(BB^*)\,\tr_{ex}(Y^{2})
\le \tr_{ex}\!\big((A^*A)\,Y\,(BB^*)\,Y\big)
\le \lambda_{\max}(A^*A)\,\lambda_{\max}(BB^*)\,\tr_{ex}(Y^{2}).
\]
\end{proof}

\begin{lemma} \label{L:Phi-bdd}
Let $\mathbb{A} = (A_1, \ldots, A_m) \in ([\alpha I, \beta I] \cap \Sigma_2)^m$ and $\omega = (w_{1}, w_{2}, \dots, w_{m}) \in \Delta_{m}$. Define $\displaystyle \Phi(X) =  \sum_{i=1}^{m} w_i  \big(I - (A_i \# X^{-1}) \big)$ for $X \in [\alpha I, \beta I]$. Then for all $Y \in SL^{2}_{ex} (\mathcal{H})$, we have
\[
\langle D\Phi(X)(Y), Y \rangle_{ex2} \geq \frac{\alpha^3}{2\beta^4} \langle Y, Y \rangle_{ex2}.
\]
\end{lemma}

\begin{proof}
Note that
\begin{align*}
& D \Phi(X)(Y) \\
&= \sum_{i=1}^{m} w_i   A_i^{1/2} \int_0^\infty e^{-t (A_i^{-1/2} X^{-1} A_i^{-1/2})^{1/2}} A_i^{-1/2} X^{-1} Y X^{-1} A_i^{-1/2} e^{-t (A_i^{-1/2} X^{-1} A_i^{-1/2})^{1/2}} \, dt \, A_i^{1/2} \\
& = \sum_{i=1}^{m} w_i \int_0^\infty B_{i}(t) Y B_{i}(t) \, dt,
\end{align*}
where $B_{i}(t) := A_i^{1/2} e^{-t (A_i^{-1/2} X^{-1} A_i^{-1/2})^{1/2}} (A_i^{-1/2} X^{-1} A_i^{-1/2} ) A_i^{1/2}$.
Since $\alpha I \leq A_{i}, X \leq \beta I$ for all $i$, we have $B_{i}(t) \geq \frac{\alpha}{\beta^{2}} e^{-\frac{t}{\alpha}}$
By Lemma \ref{l-1}, we obtain for all $Y \in SL^{2}_{ex} (\mathcal{H})$,
\begin{displaymath}
\begin{split}
\langle D\Phi(X)(Y), Y \rangle_{ex2} & = \int_0^\infty \| B_{i}(t)^{1/2} Y B_{i}(t)^{1/2} \|_{ex 2}^{2} \\
& \geq \frac{\alpha^2}{\beta^{4}} \int_0^\infty e^{-\frac{2t}{\alpha}} dt \langle Y, Y \rangle_{ex2} = \frac{\alpha^3}{2\beta^4} \langle Y, Y \rangle_{ex2}.
\end{split}
\end{displaymath}
\end{proof}


\begin{theorem} \label{T:Wass_P}
Let $\mathbb{A } =(A_1, A_2, \ldots, A_m) \in \Sigma_2^{m}$ and $\omega = (w_{1}, w_{2}, \dots, w_{m}) \in \Delta_{m}$ with $\alpha I \leq A_j \leq \beta I$ for all $j$ and some $\alpha, \beta > 0$. Then the equation \eqref{gradF} has a unique solution $X$ in $\Sigma_2 $.
\end{theorem}

\begin{proof}
Define
\[
\phi(t, X) := I - \sum_{i=1}^{m} w_{i} \big( A_{i}^{t} \# X^{-1} \big).
\]
Since \( \alpha^{t} I \leq A_j^{t} \leq \beta^{t} I \) for all \( j \) and \( t \in (0, \infty) \),
there exists \( t_{0} < 1 \) such that \( (\beta / \alpha)^{t_{0}} \leq \epsilon_{\omega}^{2} \).
By Corollary~\ref{corl}, there exists a unique
\( X_{0} \in [\, \alpha^{t_{0}} I,\, \beta^{t_{0}} I \,] \)
such that \( \phi(t_{0}, X_{0}) = 0 \).
Let \( D_{2} \phi(t, X) \) denote the Fréchet derivative of \( \phi \) with respect to its second variable.
Then, for all \( Y \in SL^{2}_{\mathrm{ex}}(\mathcal{H}) \) and any \( t \in (0, \infty) \), we have
\[
\langle D_{2} \phi(t, X)(Y), Y \rangle_{ex2}
\geq \frac{\alpha^{3t}}{2\beta^{4t}} \langle Y, Y \rangle_{ex2}.
\]
In particular, \( D_{2} \phi(t, X) \) is an invertible linear operator.
By the Implicit Function Theorem, there exist \( \theta > 0 \)
and an open neighborhood \( U \subseteq \Sigma_{2} \) containing \( X_{0} \),
together with a unique continuously differentiable function
$
g : (t_{0} - \theta,\, t_{0} + \theta) \to U
$
such that \( \phi(t, g(t)) = 0 \) for all \( t \in (t_{0} - \theta,\, t_{0} + \theta) \)
and \( g(t_{0}) = X_{0} \).
If \( 1 \in (t_{0} - \theta,\, t_{0} + \theta) \), then we are done.
Otherwise, we repeat the above process starting at \( t_{0} + \theta \).
By iterating this procedure, we eventually obtain an interval containing \( t = 1 \).
Hence, \eqref{gradF} has a solution \( X \in \Sigma_{2} \).
\end{proof}

\begin{remark}
The iterative continuation process in the above proof cannot produce shrinking intervals whose total length remains bounded. On any compact interval $t \in [t_0,1]$, the coefficient
\[
c(t) := \frac{\alpha^{3t}}{2\beta^{4t}}
\]
is continuous and strictly positive, so it attains a positive minimum value
\[
c_{\min} := \min_{t \in [t_0,1]} c(t) > 0.
\]
Hence, for all such $t$, and  $X \in [\, \alpha^{t} I,\, \beta^{t} I \,]$,
\[
\langle D_2\phi(t,X)(Y), Y\rangle_{ex2} \ge c_{\min}\|Y\|^2_{ex2},
\qquad\text{which implies}\qquad
\|D_2\phi(t,X)^{-1}\|_{ex2} \le \frac{1}{c_{\min}}.
\]
Therefore, $D_2\phi(t,X)$ remains uniformly bounded below.
Consequently, there exists a uniform step size $\theta_{\min} > 0$, independent of $t$.
After finitely many continuation steps, the interval of existence necessarily reaches $t=1$,
so the shrinking-interval scenario cannot occur.
\end{remark}

The dual space of $S^{1}(\mathcal{H})$ can be identified with $S(\mathcal{H})$ via the duality pairing
\[
\langle X, Y \rangle = \operatorname{Tr}(XY), \qquad X \in S^{1}(\mathcal{H}), \; Y \in S(\mathcal{H}).
\]
For a bounded linear operator $L : S^{1}(\mathcal{H}) \to S^{1}(\mathcal{H})$,
the (Banach) adjoint $L^{*} : S(\mathcal{H}) \to S(\mathcal{H})$ is defined by
\[
\operatorname{Tr}\big(X\,L^{*}(Y)\big)
= \operatorname{Tr}\big(L(X)\,Y\big),
\quad \forall\, X \in S^{1}(\mathcal{H}),\; Y \in S(\mathcal{H}).
\]

\begin{lemma} \label{L:iso}
Let $L : S^{1}(\mathcal{H}) \to S^{1}(\mathcal{H})$ be a bounded linear operator.
Then $L$ is an isomorphism on $S^{1}(\mathcal{H})$ if and only if its adjoint
$L^{*}$ is an isomorphism on $S(\mathcal{H})$.
\end{lemma}

\begin{theorem}\label{inv}
Let
\[
L(Y) = AYA + BYB + CYC,
\]
where $A, B, C \in \mathbb{P}$. Then $L$ is an isomorphism on $S(\mathcal{H})$.
\end{theorem}

\begin{proof}
We have $L = L^{*}$. First, observe that $L$ is a positive definite operator on the Hilbert space $S^{2}(\mathcal{H})$ with inner product
$\langle X, Y \rangle_{2} = \operatorname{Tr}(XY)$. Indeed,
\[
\langle L(X), X \rangle_{2}
= \operatorname{Tr}(A X A X) + \operatorname{Tr}(B X B X) + \operatorname{Tr}(C X C X)
> 0
\quad \text{for all } X \neq 0.
\]
Hence, $L$ is invertible on $S^{2}(\mathcal{H})$. By the spectral permanence-type Corollary of \cite{DN}, we have
\begin{equation}\label{eq:inv}
\sigma(L|_{S^{1}(\mathcal{H})}) \subseteq \sigma(L|_{S^{2}(\mathcal{H})})
\end{equation}
holds, we conclude that $0 \notin \sigma(L|_{S^{1}(\mathcal{H})})$.
Therefore, $L$ is an isomorphism on $S^{1}(\mathcal{H})$.

By Lemma \ref{L:iso}, it follows that $L$ is an isomorphism on $S(\mathcal{H})$.
\end{proof}

We can give a similar proof as Theorem \ref{T:Wass_P}  using Theorem \ref{inv} of  the following.
\begin{theorem}\label{wass_P}
Let $\mathbb{A } =(A_1, A_2, \ldots, A_m) \in \mathbb{P}^{m}$ and $\omega = (w_{1}, w_{2}, \dots, w_{m}) \in \Delta_{m}$ with $\alpha I \leq A_j \leq \beta I$ for all $j$ and some $\alpha, \beta > 0$. Then then the equation \eqref{gradF} has a unique solution $X$ in $\mathbb{P}$.
\end{theorem}

\begin{proof}
Define
\[
\phi(t, X) := I - \sum_{i=1}^{m} w_{i} \big( A_{i}^{t} \# X^{-1} \big).
\]
Since \( \alpha^{t} I \leq A_j^{t} \leq \beta^{t} I \) for all \( j \) and \( t \in (0, \infty) \),
there exists \( t_{0} < 1 \) such that \( (\beta / \alpha)^{t_{0}} \leq \epsilon_{\omega}^{2} \).
By Corollary~\ref{corol_P}, there exists a unique
\( X_{0} \in [\, \alpha^{t_{0}} I,\, \beta^{t_{0}} I \,] \)
such that \( \phi(t_{0}, X_{0}) = 0 \).
Let \( D_{2}\phi(t, X) \) denote the Fréchet derivative of \( \phi \) with respect to its second variable.
Then, for all \( Y \in S^{2}(\mathcal{H}) \) and any \( t \in (0, \infty) \), we have
\begin{equation}\label{eq:wass_P}
\langle D_{2} \phi(t, X)(Y), Y \rangle
\geq \frac{\alpha^{3t}}{2\beta^{4t}} \langle Y, Y \rangle.
\end{equation}
In particular, \( D_{2}\phi(t, X) \) is an invertible linear operator on $S^{2}(\mathcal{H})$. By Theorem \ref{inv}, \( D_{2}\phi(t, X) \) is an invertible linear operator on $S(\mathcal{H})$.
By the Implicit Function Theorem, there exist \( \theta > 0 \)
and an open neighborhood \( U \subseteq \mathbb{P} \) containing \( X_{0} \), together with a unique continuously differentiable function
$
g : (t_{0} - \theta,\, t_{0} + \theta) \to U
$
such that \( \phi(t, g(t)) = 0 \) for all \( t \in (t_{0} - \theta,\, t_{0} + \theta) \)
and \( g(t_{0}) = X_{0} \).
If \( 1 \in (t_{0} - \theta,\, t_{0} + \theta) \), then we are done.
Otherwise, we repeat the above process starting at \( t_{0} + \theta \).
By iterating this procedure, we eventually obtain an interval containing \( t = 1 \).
Hence, \eqref{gradF} has a solution \( X \in \mathbb{P} \).
\end{proof}

\begin{definition}
Let $\mathbb{A} = (A_1, A_2, \ldots, A_m) \in \mathbb{P}^m$ and $\omega = (w_{1}, w_{2}, \dots, w_{m}) \in \Delta_{m}$. We define the Wasserstein mean as a unique solution $X \in \mathbb{P}$ to the operator equation \eqref{gradF} and denote as $\Omega (\omega, \mathbb{A})$.
\end{definition}

\begin{remark}
Let $\mathbb{A} = (A_1, A_2, \ldots, A_m) \in \mathbb{P}^m$ and $\omega = (w_1, w_2, \ldots, w_m) \in \Delta_m$, where $\alpha I \leq A_j \leq \beta I$ for all $j$ and some $\alpha, \beta > 0$. The Karcher mean for positive definite operators is defined in \cite{LL} as a unique solution to the Karcher equation:\begin{equation*} \label{Karcher-eq}\sum_{i=1}^{m} w_{i} X^{1/2} \log(X^{-1/2} A_{i} X^{-1/2}) X^{1/2} = 0.\end{equation*} By utilizing the fact that the Riemannian Hessian of the potential function for the Karcher mean is positive definite on $S^2(\mathcal{H})$ (as established in Theorem 2.4 of \cite{palfia2}), one can provide an existence and uniqueness proof for the Karcher mean analogous to that of Theorem \ref{wass_P} for the Wasserstein mean.
\end{remark}

\section{No-dice theorem}

\begin{remark}\label{f}
For the map $\Phi: \Sigma_2 \to  SL^{2}_{ex} (\mathcal{H})$, we consider the 1-form given by $\omega_X(Z) = \langle \Phi(X), Z \rangle_{ex 2}$ for $X \in \Sigma_2$ and $Z \in SL^{2}_{ex} (\mathcal{H})$.
Since $\Sigma_2$ is an open flat submanifold of the vector space $SL^{2}_{ex} (\mathcal{H})$, the exterior derivative $d\omega$ (a 2-form on $\Sigma_2$) is given by

\begin{displaymath}
(d\omega)_X(Y, Z) = \langle D\Phi(X)(Y), Z \rangle_{ex 2} - \langle D\Phi(X)(Z), Y \rangle_{ex 2}
\end{displaymath}
for $Y, Z \in SL^{2}_{ex} (\mathcal{H})$.
We have
\begin{displaymath}
(d\omega)_X(Y, Z) = \langle \big( D\Phi(X) - D\Phi(X)^* \big) (Y), Z \rangle_{ex 2}.
\end{displaymath}
Since $D\Phi(X)$ is self-adjoint, therefore $d\omega = 0$.
Now, applying the Poincaré lemma, so $\omega$ is a closed 1-form on $\Sigma_2$.
Thus, there exists a smooth function $f: \Sigma_2 \to \mathbb{R}$ such that $\omega = df$.
In other words, $df_X(Z) = \langle \Phi(X), Z \rangle_{ex 2}$ for all $X \in \Sigma_2$ and $Z \in  SL^{2}_{ex} (\mathcal{H})$, which means $\Phi(X) = \nabla f(X)$.
\end{remark}

\begin{definition}[Resolvent map]
For each $X \in SL^{2}_{ex} (\mathcal{H})$, we define
\begin{equation*}\label{eq:resolventabove}
J_{\lambda }^f(X):=\underset{Y \in \Sigma_2}{\arg \min} \left\{ f(Y)+\frac{1}{2\lambda}\|X-Y\|^2_{ex2} \right\}.
\end{equation*}
\end{definition}

By Theorem 5.1 of \cite{OP}, we obtain the following.

\begin{theorem}\label{T:3}
Consider a function $f$ as in Remark \ref{f},
$f:X\to (-\infty,\infty]$ is  a convex, continuous function
and let $ \Sigma_2 \subset SL^{2}_{ex} (\mathcal{H})$ be a convex set.
Take a positive sequence $\{\lambda_k\}_{k \ge 1}$ with $\sum_{k=1}^\infty\lambda_k=+\infty$.
Fix an arbitrary starting point $X_0\in \Sigma_2  $ and put
\[ X_{k}:=J_{\lambda_k}^f(X_{k-1}), \qquad k \ge 1. \]
Then we have $\lim_{k \to \infty} X_k= \Omega (\omega, \mathbb{A})$.
\end{theorem}

Since the linear operator $D\Phi(X)(\cdot):S(\mathcal{H})\mapsto S(\mathcal{H})$ satisfies \eqref{eq:inv}, by Lemma~\ref{L:iso} it has spectrum $\sigma(D\Phi(X)|_{S(\mathcal{H})})=\sigma(D\Phi(X)|_{S^{1}(\mathcal{H})})\subseteq \sigma(D\Phi(X)|_{S^{2}(\mathcal{H})})$. Then \eqref{eq:wass_P} ensures that $\sigma(D\Phi(X)|_{S(\mathcal{H})})\subseteq (c,\infty)$, for example, for $c=\frac{\alpha^{3}}{2\beta^{4}}$ when $0<\alpha I \leq X,A_j \leq \beta I$ is satisfied.

The \emph{logarithmic norm} of a linear operator $A$ with respect to an operator norm $\|\cdot\|$ on a Banach space is defined and satisfying
\begin{equation*}
\mu_{\|\cdot\|}(A):=\sup_{t>0}\frac{\log\left\|e^{tA}\right\|}{t}=\lim_{t\to 0+}\frac{\log\left\|e^{tA}\right\|}{t}=\lim_{t\to 0+}\frac{\|I+tA\|-1}{t}
\end{equation*}
which describes the largest growth rate of the generated semigroup, see for example \cite{DaleckiiKrein,PerovKostrub}. The other essential quantity describing the asymptotic growth of semigroups is the \emph{spectral abscissa} defined as $\Re\sigma(A):=\max\{\Re (\sigma(A))\}$, the largest real part of the spectrum. It crucially satisfies
\begin{equation*}
\Re\sigma(A)=\inf_{t>0}\frac{\log\left\|e^{tA}\right\|}{t}=\lim_{t\to \infty}\frac{\log\left\|e^{tA}\right\|}{t}
\end{equation*}
and describes the smallest upper-bound exponential growth rate \cite{DaleckiiKrein,PerovKostrub}. In particular, for any $\epsilon>0$ there exists an $M(\epsilon)>0$ such that
\begin{equation*}
\left\|e^{tA}\right\|\leq M(\epsilon)e^{t(\Re\sigma(A)+\epsilon)}
\end{equation*}
for all $t\geq 0$, see for instance \cite{DaleckiiKrein}. The following renorming procedure is known \cite{PerovKostrub}, but we provide a proof for completeness.

\begin{lemma}\label{L:equiNorm}
Given $\epsilon>0$ and $0<\alpha I \leq X,A_j \leq \beta I$ with $X,A_j\in S(\mathcal{H})$, there exists an equivalent norm $\|\cdot\|_{\Re\sigma(D\Phi(X))+\epsilon}$ on $S(\mathcal{H})$ such that the logarithmic norm satisfies
\begin{equation}\label{eq:L:equiNorm}
\mu_{\|\cdot\|_{\Re\sigma(D\Phi(X))+\epsilon}}(D\Phi(X))\leq \Re\sigma(D\Phi(X))+\epsilon.
\end{equation}
\end{lemma}

\begin{proof}
We introduce a new (candidate) norm as
\begin{equation*}
\|v\|_{\lambda}:=\int_0^\infty\|e^{tD\Phi(X)}v\|e^{-t\lambda}dt
\end{equation*}
for any $\lambda>\Re\sigma(D\Phi(X))$. By the properties of strongly continuous semigroups \cite{DaleckiiKrein} $\exists M(\epsilon)>0$ such that
\begin{equation*}
\|e^{tD\Phi(X)}\|\leq M(\epsilon)e^{t(\Re\sigma(D\Phi(X))+\epsilon)}
\end{equation*}
for all $0\leq t\leq \infty$. We claim that $\|\cdot\|_{\lambda}$ is well defined and as a norm is equivalent to $\|\cdot\|$. First we have that
\begin{equation*}
\|v\|_{\lambda}\leq \int_0^\infty M(\epsilon)e^{t(\Re\sigma(D\Phi(X))+\epsilon)}\|v\|e^{-t\lambda}dt=\frac{M(\epsilon)}{\lambda-(\Re\sigma(D\Phi(X))+\epsilon)}\|v\|
\end{equation*}
which establishes that $\|\cdot\|_{\lambda}$ is finite and it is bounded above by $\|\cdot\|$.
Next, using the uniform estimate \(\|D\Phi(X))\|\le \theta\),
\[
\|e^{tD\Phi(X))}-I\|
\le
\sum_{n=1}^\infty \frac{t^n\|D\Phi(X))\|^n}{n!}
\le
\sum_{n=1}^\infty \frac{(t\theta)^n}{n!}
=
e^{\theta t}-1.
\]
Now choose \(\delta>0\) so that \(e^{\theta\delta}-1\le \frac12\). Then for every \(t\in[0,\delta]\),
\[
\|e^{tD\Phi(X))}-I\|\le \frac12,
\qquad X \in [\alpha I, \beta I].
\]
Consequently, for every \(v \),
\[
\|e^{tD\Phi(X))}v\|
\ge
\|v\|-\|(e^{tD\Phi(X))}-I)v\|
\ge
\Bigl(1-\|e^{tD\Phi(X))}-I\|\Bigr)\|v\|
\ge
\frac12\|v\|,
\qquad 0\le t\le \delta.
\]
Therefore, the lower bound on $\|\cdot\|_{\lambda}$ follows from
\[
\|v\|_{\lambda}
\ge
\int_0^\delta  \|e^{tD\Phi(X)}v\|\,e^{-\lambda t}\,dt
\ge
\frac12\int_0^\delta e^{-\lambda t}\,dt\,\|v\|.
\]
Finally, by change of variables we have
\begin{equation}\label{eq1:L:equiNorm}
\begin{split}
\|e^{tD\Phi(X)}v\|_{\Re\sigma(D\Phi(X))+\epsilon}&=e^{t(\Re\sigma(D\Phi(X))+\epsilon)}\int_t^\infty\|e^{sD\Phi(X)}v\|e^{-s(\Re\sigma(D\Phi(X))+\epsilon)}ds\\
&\leq e^{t(\Re\sigma(D\Phi(X))+\epsilon)}\int_0^\infty\|e^{sD\Phi(X)}v\|e^{-s(\Re\sigma(D\Phi(X))+\epsilon)}ds\\
&=e^{t(\Re\sigma(D\Phi(X))+\epsilon)}\|v\|_{\Re\sigma(D\Phi(X))+\epsilon}
\end{split}
\end{equation}
which proves that \eqref{eq:L:equiNorm} holds.
\end{proof}

\begin{proposition}\label{P:contiNorm}
The family of equivalent norms $\|\cdot\|_{\Re\sigma(D\Phi(X))+\epsilon}$ constructed in Lemma~\ref{L:equiNorm} vary continuously with $X$.
\end{proposition}

\begin{proof}
This follows from the perturbation estimate of semigroups such as Theorem VII.1.3 \cite{DaleckiiKrein}. That is, $\left\|e^{tA}\right\|\leq Me^{\omega t}$ implies that $\left\|e^{t(A+B)}\right\|\leq Me^{(\omega+M\|B\|)t}$ for small enough $\|B\|$ which follows by taking the norm of the integral version of the variation of constants formula and then applying Gr\"onwall's lemma, or bounding the terms in the resulting Dyson-Phillips series of the dynamical system. Indeed choosing $A:=D\Phi(X)$ and $B:=D\Phi(Y)-D\Phi(X)$ for $Y$ near $X$ the previous estimate concludes the assertion.
\end{proof}

\begin{definition}[The induced length metric $d_\epsilon$]
The continuity property of the family of norms in Proposition~\ref{P:contiNorm} ensures that $(S(\mathcal{H}),\|\cdot\|_{\Re\sigma(D\Phi(X))+\epsilon})$ is a reversible Banach-Finsler manifold with induced length metric
\begin{equation*}
d_\epsilon(A,B):=\inf_{\gamma\in\mathcal{C}^1:\gamma(0)=A, \gamma(1)=B}\int_0^1\|\dot{\gamma}(t)\|_{\Re\sigma(D\Phi(\gamma(t)))+\epsilon}dt.
\end{equation*}
\end{definition}

\begin{theorem}\label{T:contractionBanachFinsler}
Let
\begin{align*}
    u_i'(t) &= - \Phi(u_i(t)), \\
    u_i(0) &\in\mathbb{P}.
\end{align*}
Then, for all $t \geq 0$, we have
\begin{equation}\label{eq1:T:contractionBanachFinsler}
d_\epsilon(u_1(t), u_2(t)) \leq e^{(\sup\{\Re\sigma(-D\Phi(X)):X\in S\}+\epsilon) t} d_\epsilon(u_1(0), u_2(0))
\end{equation}
where $u_1(t), u_2(t)\in S$ for some large enough $S\subseteq\mathbb{P}$.
\end{theorem}

\begin{proof}
We follow the ideas of the proofs of Theorem 1 \& 2 in \cite{SrinivasanSlotine}. Given a $\mathcal{C}^1$-solution $u(t)$ of the initial value problem of the assertion, any perturbation $\delta u(t)\in T_{u(t)}\mathbb{P}$ at $u(t)$ evolves according to the linear dynamics
\begin{equation*}
\frac{d}{dt}\delta u= -D\Phi(u)(\delta u),
\end{equation*}
see for instance Sections VII.1-3. and within Equation (1.11) in \cite{DaleckiiKrein}. Then we have by \eqref{eq1:L:equiNorm} and \eqref{eq:L:equiNorm} that
\begin{equation*}
\begin{split}
\frac{d^+}{dt}\|\delta u(t)\|_{\Re\sigma(-D\Phi(u))+\epsilon}&\leq (\Re\sigma(D\Phi(u))+\epsilon)\|\delta u(t)\|_{\Re\sigma(D\Phi(u))+\epsilon}\\
&\leq (\sup\{\Re\sigma(-D\Phi(X)):X\in S\}+\epsilon)\|\delta u(t)\|_{\Re\sigma(D\Phi(u))+\epsilon}
\end{split}
\end{equation*}
and Gr\" onwall's lemma gives
\begin{equation*}
\|\delta u(t)\|_{\Re\sigma(-D\Phi(u(t)))+\epsilon}\leq e^{(\sup\{\Re\sigma(-D\Phi(X)):X\in S\}+\epsilon)(t-t_0)}\|\delta u(t_0)\|_{\Re\sigma(D\Phi(u(t_0)))+\epsilon}.
\end{equation*}
Then first integrating the above along a $\mathcal{C}^1$-curve $\gamma_t(s)$ with $\gamma_t(0)=u_1(t), \gamma_t(1)=u_2(t)$ and then taking infimum in all such $\gamma_t$ yields \eqref{eq1:T:contractionBanachFinsler}.
\end{proof}

Let $\mathbb{T} = [\underline{\alpha} I, \overline{\beta} I]$ for some constants $\underline{\alpha}, \overline{\beta}$ such that $0 < \underline{\alpha} < \alpha < \beta < \overline{\beta}$. The following is immediate.

\begin{lemma} \label{L:Phi-specific-bounds}
The map \(\Phi: \mathbb{T} \to S(\mathcal H)\) appeared in Lemma \ref{L:Phi-bdd} is \( \mathcal{C}^1 \), bounded, and globally Lipschitz in the operator norm.
More precisely, for all \( X,Y \in \mathbb{T} \), there exist $B_{\Phi}, L_{\Phi}$ such that
\[
\| \Phi(X) \| \le B_{\Phi},
\]
and
\[
\| \Phi(X) - \Phi(Y) \| \le L_{\Phi} \| X - Y \|.
\]
\end{lemma}

\begin{lemma} \label{L:flow-invariance-S1}
The first-order system
\[
u'(t) = -\Phi(u(t)),
\]
leaves \( \mathbb{T} \) invariant.
\end{lemma}

\begin{proof}
Set
\[
F(X) := -\Phi(X) = \sum_{i=1}^m w_i \bigl( A_i \# X^{-1} - I \bigr).
\]
Then $F$ is Lipschitz on $\mathbb{T}$ by Lemma \ref{L:Phi-specific-bounds}.
Since $\mathbb{T}$ is closed and
convex, by \cite[Remark 2.4]{GQ} and \cite[Proposition 2.5]{GQ},
it is enough to prove that
$F(X) \in T_{\mathbb{T}}(X)$ for any $X \in \mathbb{T}$,
where
$T_{\mathbb{T}}(X) = \{ V: \exists \lambda > 0, \ X + \lambda V \in \mathbb{T} \}$.
Equivalently, it is enough to show that for every $X \in \mathbb{T}$ there exists $\lambda > 0$ such that
\[
\underline{\alpha} I \leq X + \lambda F(X) \leq \overline{\beta} I.
\]
We obtain
$\sqrt{\alpha} \, X^{-1/2} \leq A_i \# X^{-1} \leq \sqrt{\beta} \, X^{-1/2}.$
Taking the weighted sum gives
\begin{equation*}\label{eq:F-order-bound-tangent}
\sqrt{\alpha} \, X^{-1/2} - I \leq F(X) \leq \sqrt{\beta} \, X^{-1/2} - I.
\end{equation*}

Let $X \in \mathbb{T}$. Then $\sigma(X) \subseteq [\underline{\alpha}, \overline{\beta}]$. Choose $0 < \lambda \leq 2\underline{\alpha}$.
For $t \in [\underline{\alpha}, \overline{\beta}]$, we have
$t + \lambda (\sqrt{\underline{\alpha}} \, t^{-1/2} - 1)
= \underline{\alpha} + (t - \underline{\alpha}) \left( 1 - \frac{\lambda}{\sqrt{t}(\sqrt{t}+\sqrt{\underline{\alpha}})} \right)$.
Since
$\sqrt{t}(\sqrt{t}+\sqrt{\underline{\alpha}}) \geq 2\underline{\alpha} \geq \lambda$,
it follows that
$t + \lambda (\sqrt{\underline{\alpha}} \, t^{-1/2} - 1) \geq \underline{\alpha}$.
Therefore,
$X + \lambda F(X) \geq X + \lambda (\sqrt{\underline{\alpha}} \, X^{-1/2} - I) \geq \underline{\alpha} I$.

Similarly, for $t \in [\underline{\alpha},\overline{\beta}]$,
$t + \lambda (\sqrt{\overline{\beta}} \, t^{-1/2} - 1)
= \overline{\beta} - (\overline{\beta}-t) \left(
1 - \frac{\lambda}{\sqrt{t}(\sqrt{\overline{\beta}}+\sqrt{t})}
\right)$.
Since $2 \underline{\alpha} \leq \sqrt{\underline{\alpha}} (\sqrt{\overline{\beta}} + \sqrt{\underline{\alpha}})$, we have
$\sqrt{t}(\sqrt{\overline{\beta}}+\sqrt{t}) \geq
\sqrt{\underline{\alpha}}(\sqrt{\overline{\beta}}+\sqrt{\underline{\alpha}}) \geq \lambda$ and
$t + \lambda (\sqrt{\overline{\beta}} \, t^{-1/2} - 1) \leq \overline{\beta}$.
Thus,
$X + \lambda F(X) \leq X + \lambda (\sqrt{\overline{\beta}} \, X^{-1/2} - I)
\leq \overline{\beta} I$.
We obtain
$\underline{\alpha} I \leq X + \lambda F(X) \leq \overline{\beta} I$.

Consequently, if $X \in \mathbb{T}$ and
$u'(t) = -\Phi(u(t))$ with $u(0) = X$,
then $u(t) \in \mathbb{T}$ for all $t \geq 0$.
\end{proof}

Theorem \ref{T:contractionBanachFinsler}, together with Lemma
\ref{L:flow-invariance-S1}, gives another proof of Theorem~\ref{wass_P}, i.e.

\begin{corollary}\label{uniq}
The existence and uniqueness of the point \(\Omega(\omega,\mathbb A)\) holds, it is the unique solution of the equation $\Phi(X) = 0$.
\end{corollary}

\begin{lemma}\label{L:G0-Taylor}
Let $G_0(Y,X) := I - (X \# Y^{-1})$ for $X, Y \in [\alpha I,\beta I]$.
Then there exists a constant \( C_{\alpha,\beta} > 0 \), depending only on $\alpha$ and $\beta$,
such that
\[
G_0(Y,X) = D_1G_0(X,X)(Y-X) + R_0(Y,X),
\qquad
\| R_0(Y,X) \| \le C_{\alpha,\beta} \| Y-X \|^2
\]
for all \( X,Y \in [\alpha I,\beta I] \).
\end{lemma}

\begin{proof}
We obtain
\begin{equation*}
\begin{split}
&D_1G_0(Y,X)(Z)\\
&=
X^{1/2} \int_0^\infty
e^{-t(X^{-1/2}Y^{-1}X^{-1/2})^{1/2}} X^{-1/2}Y^{-1}ZY^{-1}X^{-1/2} e^{-t(X^{-1/2}Y^{-1}X^{-1/2})^{1/2}} \, dt \, X^{1/2}.
\end{split}
\end{equation*}
Evaluating at \(Y=X\) gives
$D_1G_0(X,X)(Z)
= \int_0^\infty e^{-tX^{-1}}X^{-1}ZX^{-1}e^{-tX^{-1}}\,dt$.
By Taylor's formula with integral remainder yields
\[
G_0(Y,X) = G_0(X,X) + D_1G_0(X,X)(Y-X)
+ \int_0^1 (1-t) \, D_1^2G_0(X+t(Y-X), X)(Y-X)^2 \, dt,
\]
that is,
$G_0(Y,X) = D_1G_0(X,X)(Y-X) + R_0(Y,X)$ since $G_{0}(X,X) = 0$,
where
$R_0(Y,X)
:= \int_0^1 (1-t) \, D_1^2G_0(X+t(Y-X), X)(Y-X)^2 \, dt$.
By Lemma \ref{L:Phi-specific-bounds} there exists $C_{\alpha,\beta} > 0$ such that
\[
\| D_1^2G_0(X+t(Y-X), X) \|
\le
C_{\alpha,\beta}
\]
for all \( X,Y \in [\alpha I,\beta I] \). Therefore
\[
\| R_0(Y,X) \|
\le
\frac{1}{2}
\sup_{0 \le t \le 1} \| D^2G_0(X+t(Y-X), X) \|
\, \| Y-X \|^2
\le
C_{\alpha,\beta} \| Y-X \|^2.
\]
This proves the claim.
\end{proof}

\begin{lemma}\label{L:G-Taylor}
Let $X, Y \in [\alpha I, \beta I]$. Define
\begin{equation}\label{eq:def-G-normalized}
G(Y,X) :=
X \bigl( I - (X \# Y^{-1}) \bigr) + \bigl( I - (X \# Y^{-1}) \bigr) X.
\end{equation}
Then
\begin{equation*}\label{eq:G-normalized-expansion}
G(Y,X) = Y-X + R(Y,X),
\qquad
\| R(Y,X) \| \le \widetilde C_{\alpha,\beta} \| Y-X \|^2.
\end{equation*}
\end{lemma}

\begin{lemma}
\label{L:uniform-local-consistency}
Under the assumptions of Lemma \ref{L:flow-invariance-S1},
there exist constants \( h_{\mathrm{loc}} > 0 \) and \( C_{\mathrm{loc}} > 0 \) such that, for every \( X \in \mathbb{T} \) and every \( 0 < h \le h_{\mathrm{loc}} \), the normalized resolvent equation
\[
0 = \Phi(Y) + \frac{1}{h} \, G(Y,X)
\]
has a solution \( Y = J_h(X) \) satisfying
\[
\| J_h(X) - (X-h\Phi(X)) \| \le C_{\mathrm{loc}}h^2.
\]
\end{lemma}

\begin{proof}
Let $U := \left[ \frac{\underline{\alpha}}{2}I, \, 2\overline{\beta}I \right]$.
%
We have
\[
G(Y,X) = Y-X + R(Y,X),
\]
where
$\| R(Y,X) \| \le \frac{C}{2} \| Y-X \|^2$ and
$\| D_1R(Y,X) \| \le C \| Y-X \|$ for
$Y \in U, X \in \mathbb{T}$.
Fix \( X \in \mathbb{T} \) and \( h > 0 \). Define
\[
T_{h,X}(Y) := X - h\Phi(Y) - R(Y,X).
\]
Then \(Y\) solves
$0 = \Phi(Y) + \frac1h \, G(Y,X)$
if and only if \(Y\) is a fixed point of \( T_{h,X} \).
Set
$R := 2 B_\Phi$.
Let \( B(X,Rh) \) be a closed ball centered at \(X\) of radius \(Rh\).
If \( Y \in B(X,Rh) \), then
\[
\| T_{h,X}(Y)-X \|
\le h \| \Phi(Y) \| + \| R(Y,X) \|
\le h B_\Phi + \frac{C}{2} R^2h^2.
\]
Hence, provided
$0 < h \le \frac{1}{CR}$,
we have
\[
\| T_{h,X}(Y) - X \|
\le hB_\Phi + \frac{R}{2}h
= Rh.
\]
Thus, \(T_{h,X}\) maps \(B(X,Rh)\) into itself.
Next, for \( Y_1,Y_2 \in B(X,Rh) \),
$$
\| T_{h,X}(Y_1) - T_{h,X}(Y_2) \|
\le h L_\Phi \| Y_1-Y_2 \| + \sup_{\| Y-X \| \le Rh} \| D_1R(Y,X) \| \, \| Y_1-Y_2 \|.
$$
Since \( \| D_1R(Y,X) \| \le CRh \) on this ball, we obtain
\[
\| T_{h,X}(Y_1) - T_{h,X}(Y_2) \|
\le h (L_\Phi + CR) \| Y_1-Y_2 \|.
\]
Therefore, if
$0 < h \le \frac{1}{2(L_\Phi + CR)}$,
then \(T_{h,X}\) is a contraction with Lipschitz constant at most \(1/2\).

Finally, to ensure \( B(X,Rh) \subset U \) uniformly for \( X \in \mathbb{T} \), it requires $Rh \le \frac{\underline{\alpha}}{2}$.
By the preceding argument with $CR \leq 2(L_\Phi + CR)$, we choose
$h_{\mathrm{loc}}
:=
\min \left\{ \frac{1}{2(L_\Phi + CR)}, \, \frac{\underline{\alpha}}{2R} \right\}$.
Thus, with such $h_{\mathrm{loc}}$,
the map \(T_{h,X}\) is a contraction from \(B(X,Rh)\) to itself for every \( X \in \mathbb{T} \) and \( 0 < h \le h_{\mathrm{loc}} \).
By the Banach fixed point theorem, there exists a unique
fixed point \( J_h(X) \in B(X,Rh) \subset U \), and hence
$\| J_h(X) - X \| \le Rh$.

Now,
$J_h(X) - (X - h \Phi(X))
= -h \bigl( \Phi(J_h(X)) - \Phi(X) \bigr) - R(J_h(X),X)$.
Therefore,
\[
\| J_h(X) - (X - h \Phi(X)) \|
\le
h L_\Phi \| J_h(X) - X \| + \frac{C}{2} \| J_h(X) - X \|^2
\le
\left( L_\Phi R + \frac{C}{2} R^2 \right) h^2.
\]
Thus the claim holds with
$C_{\mathrm{loc}} := L_\Phi R + \frac{C}{2} R^2$.
\end{proof}

\begin{lemma}
\label{L:uniform-trapping-normalized}
Under the assumptions of Lemma \ref{L:flow-invariance-S1},
there exists \( h_0 > 0 \) such that
\[
J_h(\mathbb{T}) \subseteq \mathbb{T}
\qquad\text{for all } 0 < h \le h_0.
\]
\end{lemma}

\begin{proof}
Define the map
$E_h(X) := X - h \Phi(X)
= X - hI + h \sum_{i=1}^m w_i (A_i \# X^{-1})$.
Since \( \alpha I \le A_i \le \beta I \), we have
$\sqrt{\alpha} \, X^{-1/2}
\le
\sum_{i=1}^m w_i (A_i \# X^{-1})
\le
\sqrt{\beta} \, X^{-1/2}$.
Therefore,
$$
X - hI + h \sqrt{\alpha} \, X^{-1/2} \leq E_h(X) \leq X - hI + h \sqrt{\beta} \, X^{-1/2}.
$$

We first derive a uniform inward estimate for \(E_h\) on \( \mathbb{T} \).
Consider the scalar functions
\[
f_h(x) := x - \underline{\alpha} - h + h \sqrt{\frac{\alpha}{x}},
\qquad
g_h(x) := \overline{\beta} - x + h - h \sqrt{\frac{\beta}{x}},
\qquad
\textrm{for} \ x \in [\underline{\alpha},\overline{\beta}].
\]
Their derivatives are
$
f_h'(x) = 1 - \frac{h\sqrt{\alpha}}{2x^{3/2}},
\
g_h'(x) = -1 + \frac{h\sqrt{\beta}}{2x^{3/2}}.
$
If
$0 < h \le h_E := \frac{2\underline{\alpha}^{3/2}}{\sqrt{\beta}}$,
then
$f_h'(x) \ge 0, \ g_h'(x) \le 0$
for all $x \in[\underline{\alpha},\overline{\beta}]$.
Hence \(f_h\) is increasing and \(g_h\) is decreasing on
\([\underline{\alpha},\overline{\beta}]\), and therefore
\[
f_h(x) \ge f_h(\underline{\alpha})
=
h \left( \sqrt{\frac{\alpha}{\underline{\alpha}}} - 1 \right),
\]
\[
g_h(x) \ge g_h(\overline{\beta})
=
h \left( 1 - \sqrt{\frac{\beta}{\overline{\beta}}} \right).
\]
Setting
$\kappa_- := \sqrt{\frac{\alpha}{\underline{\alpha}}} - 1 > 0,
\
\kappa_+ := 1 - \sqrt{\frac{\beta}{\overline{\beta}}} > 0$, and
$\kappa := \min \{ \kappa_-, \kappa_+ \} > 0$,
we obtain
\begin{equation} \label{E:bounds-E_h}
\underline{\alpha} I + \kappa h \, I \leq E_h(X) \leq \overline{\beta} I - \kappa h \, I,
\qquad
X \in \mathbb{T}, \ \ 0 < h \le h_E.
\end{equation}
Thus, the explicit Euler map sends \(\mathbb{T}\) strictly into its interior, with a uniform margin of size \( \kappa h \).

Now let \( 0 < h \le \min\{h_E, h_{\mathrm{loc}}\} \), where $h_{\mathrm{loc}}$ is appeared in the proof of Lemma \ref{L:uniform-local-consistency}.
By Lemma \ref{L:uniform-local-consistency},
$\| J_h(X) - E_h(X) \| \le C_{\mathrm{loc}} h^2$ for $X \in \mathbb{T}$.
We have
\[
-C_{\mathrm{loc}} h^2 I \le J_h(X)-E_h(X) \le C_{\mathrm{loc}} h^2 I.
\]
Hence, by \eqref{E:bounds-E_h}
\[
J_h(X) \ge E_h(X) - C_{\mathrm{loc}} h^2 I \ge
\underline{\alpha} I + \kappa h \, I - C_{\mathrm{loc}} h^2 I,
\]
and similarly
\[
J_h(X) \le E_h(X) + C_{\mathrm{loc}} h^2 I \le
\overline{\beta} I - \kappa h \, I + C_{\mathrm{loc}} h^2 I.
\]
If
$0 < h \le \frac{\kappa}{2C_{\mathrm{loc}}}$,
then
$\kappa h - C_{\mathrm{loc}} h^2 \ge \frac{\kappa}{2} h > 0$.
Therefore,
\[
J_h(X) \ge \underline{\alpha} I + \frac{\kappa}{2} h \, I > \underline{\alpha}I,
\qquad \textrm{and} \qquad
J_h(X) \le \overline{\beta} I - \frac{\kappa}{2} h \, I < \overline{\beta}I.
\]
Thus, $J_h(X) \in \mathbb{T}$
for all $X \in \mathbb{T}$ and $0 < h \le h_0$,
where
$h_0 := \min \left\{
\frac{2\underline{\alpha}^{3/2}}{\sqrt{\beta}},
\ h_{\mathrm{loc}},
\ \frac{\kappa}{2C_{\mathrm{loc}}}
\right\}$.
\end{proof}

The above proof shows that we may as well \emph{define} the resolvent $J_h(X)=Y$ by requiring only
\begin{equation*}
0 = \Psi(X) + \frac1h \, (Y-X) + O(h^2)
\end{equation*}
for small enough $\lambda>0$ where the error term $O(h^2)$ is assumed to be uniform for all $X\in\mathbb{T}$ and then still obtain the invariance of $\mathbb{T}$, that is:
\begin{corollary}[Freedom of resolvents]
There exists \( h_0 > 0 \) such that
\[
J_h(\mathbb{T}) \subseteq \mathbb{T}
\qquad\text{for all } 0 < h \le h_0,
\]
if $0 = \Psi(X) + \frac1h \, (Y-X) + O(h^2),$ with error term $O(h^2)$ assumed to be uniform for any $X\in\mathbb{T}$.
\end{corollary}

\begin{proposition}[Resolvent steps]
\label{P:concrete-resolvent-consistency}
Fix \( X \in \mathbb{T} \), let \(0 < h \leq h_0 \), and let \( Y_h \in \mathbb{T} \) be a
solution of the following equation:
\begin{equation}\label{eq:normalized-BW-resolvent}
0 = \Phi(Y) + \frac{1}{h} \, G(Y,X),
\end{equation}
where \(G\) is given by \eqref{eq:def-G-normalized}. Then
\begin{equation}\label{eq:normalized-BW-expansion}
Y_h = X - h \, \Phi(X) + O(h^2).
\end{equation}
\end{proposition}

\begin{proof}
Set
$\Delta_h := Y_h - X$.
By Lemma \ref{L:G-Taylor}, there exists a constant
\( \widetilde C_{\underline{\alpha},\overline{\beta}} > 0 \) such that
\[
G(Y_h,X) = \Delta_h + R(Y_h,X),
\qquad
\| R(Y_h,X) \| \le \widetilde C_{\underline{\alpha},\overline{\beta}} \| \Delta_h \|^2.
\]
Substituting this into \eqref{eq:normalized-BW-resolvent}, we obtain
\[
0 = \Phi(Y_h) + \frac{1}{h} \bigl( \Delta_h + R(Y_h,X) \bigr)
\]
so
\begin{equation}\label{eq:Delta-identity}
\Delta_h = -h \, \Phi(Y_h) - R(Y_h,X).
\end{equation}

Since \(\Phi\) is bounded on \( [\underline{\alpha} I, \overline{\beta} I] \),
say \( \| \Phi(Z) \| \le B_{\Phi} \), we get
from \eqref{eq:Delta-identity}
$\| \Delta_h \|
\le B_{\Phi} h + \widetilde C_{\underline{\alpha},\overline{\beta}} \| \Delta_h \|^2$.
Because \( Y_h \to X \), we have \( \| \Delta_h \| \to 0 \).
Therefore, for small \(h\),
$\widetilde C_{\underline{\alpha},\overline{\beta}} \| \Delta_h \| \le \frac12$,
and so
$\widetilde C_{\underline{\alpha},\overline{\beta}} \| \Delta_h \|^2 \le \frac12 \| \Delta_h \|$.
Thus
$\| \Delta_h \| \le B_{\Phi} h + \frac12 \| \Delta_h \|$,
which implies
$\| \Delta_h \| \le 2B_{\Phi} h$.
Hence,
$\Delta_h = O(h)$ gives
$R(Y_h,X) = O(\| \Delta_h \|^2) = O(h^2)$.
Since \(\Phi\) is Lipschitz, we also have
$\Phi(Y_h) = \Phi(X) + O(\| \Delta_h \|) = \Phi(X) + O(h)$.
Substituting these two estimates into \eqref{eq:Delta-identity}, we obtain
\[
\Delta_h = -h \bigl( \Phi(X) + O(h) \bigr) + O(h^2)
= -h \, \Phi(X) + O(h^2).
\]
Since \( \Delta_h = Y_h-X \), this proves
$Y_h = X - h \, \Phi(X) + O(h^2)$.
\end{proof}

Let $\{\lambda_k\}_{k \ge 1}$ be a positive sequence with $\lambda_1 \leq h_0$, $\sum_{k=1}^\infty \lambda_k = +\infty$ and $\sum_{k=1}^\infty \lambda_k^2 < +\infty$.

\begin{proposition}[Forward Euler step method]\label{P:monotoneResolvent}
Let $X_* \in [\alpha I,\beta I]$ be such that $\Phi(X_*)=0$, $a := \sup\{ \Re\sigma(-D\Phi(X)): X \in \mathbb{T} \} + \epsilon$ and
\( (X_k)_{k\ge1} \subset \mathbb{T} \) satisfies
\begin{equation}\label{Resolvent}
X_{k+1} = X_k - \lambda_{k+1} \Phi(X_k) + r_k,
\qquad
\| r_k \| \le C_r \lambda_{k+1}^2.
\end{equation}
Then
\begin{equation*}\label{eq1:P:monotoneResolvent}
d_\epsilon(X_*,X_{k+1}) \leq \left( 1 + a \lambda_{k+1} \right) d_\epsilon(X_*,X_{k}) + O \left( \lambda_{k+1}^2 \right)
\end{equation*}
and $d_\epsilon(X_*,X_{k}) \to 0$ as $k \to \infty$.
\end{proposition}

\begin{proof}
Let \(u_k\) be the solution of
\[
u_k'(t) = -\Phi(u_k(t)), \qquad u_k(0) = X_k.
\]
Since \( X_k \in \mathbb{T} \),
we have $u_k(t) \in \mathbb{T}$ for $0 \le t \le \lambda_{k+1}$.
Using the integral form of the exact flow,
\[
u_k(t) = X_k - \int_0^t \Phi(u_k(s)) \, ds,
\]
and we can write
\[
u_k(t) - X_k + t \Phi(X_k)
=
-\int_0^t \bigl( \Phi(u_k(s)) - \Phi(X_k) \bigr) \, ds.
\]
Since
$\displaystyle \| u_k(s) - X_k \|
\le \int_0^s \| \Phi(u_k(\tau)) \| \, d\tau
\le B_\Phi s$,
we have by the Lipschitz bound for \(\Phi\),
\[
\| \Phi(u_k(s))-\Phi(X_k) \|
\le
L_\Phi \| u_k(s)-X_k \|
\le
L_\Phi B_\Phi s,
\]
so
\[
\| u_k(t) - X_k + t \Phi(X_k) \|
\le \int_0^t L_\Phi B_\Phi s \, ds
= \frac{L_\Phi B_\Phi}{2} t^2.
\]
Setting \( t = \lambda_{k+1} \), we obtain
\[
u_k(\lambda_{k+1})
= X_k - \lambda_{k+1} \Phi(X_k) + \tau_k,
\qquad
\| \tau_k \| \le \frac{L_\Phi B_\Phi}{2} \lambda_{k+1}^2.
\]

Comparing this with \eqref{Resolvent},
we get $X_{k+1} - u_k(\lambda_{k+1}) = r_k - \tau_k$,
and hence
\[
\| X_{k+1} - u_k(\lambda_{k+1}) \|
\le
\left( C_r + \frac{L_\Phi B_\Phi}{2} \right) \lambda_{k+1}^2.
\]
Again, since \( u_k(\lambda_{k+1}), X_{k+1} \in \mathbb{T} \) and $d_{\epsilon}$ is equivalent to $\| \cdot \|$, there exists $M > 0$ so that
\[
d_\epsilon \bigl( u_k(\lambda_{k+1}), X_{k+1} \bigr)
\le
M \| X_{k+1} - u_k(\lambda_{k+1}) \|
\le
M \left( C_r + \frac{L_\Phi B_\Phi}{2} \right) \lambda_{k+1}^2.
\]
Since \(X_*\) is a solution of $\Phi(X) = 0$, the contraction estimate in Theorem \ref{T:contractionBanachFinsler} gives
\[
d_\epsilon(X_*,u_k(\lambda_{k+1})) \le e^{a \lambda_{k+1}} d_\epsilon(X_*,X_k).
\]
Hence, by the triangle inequality,
\[
d_\epsilon(X_*,X_{k+1})
\le e^{a \lambda_{k+1}} d_\epsilon(X_*,X_k) + O(\lambda_{k+1}^2).
\]
Using $e^{a \lambda_{k+1}} \le 1 + a \lambda_{k+1} + O(\lambda_{k+1}^2)$
and the uniform boundedness of \(d_\varepsilon(X_*,X_k)\) on \( \mathbb{T} \),
we conclude that
\[
d_\epsilon(X_*,X_{k+1})
\le
(1 + a \lambda_{k+1}) d_\epsilon(X_*,X_k) + O(\lambda_{k+1}^2).
\]
Finally, given the specific ${l}^1,{l}^2$ properties of the sequence $\{\lambda_k\}_{k \ge 1}$, from the above one can now obtain from Lemma 7.3 in \cite{HallHeyde} that $d_\epsilon(X^*,X_{k}) \to 0$ as $k \to \infty$.
\end{proof}

Let $\mathbb A = (A_1, \dots, A_m) \in ([\alpha I,\beta I])^m$ and $\omega = (w_1, \dots, w_m) \in \Delta_m$.
Set $\Psi_i(X) := m w_i(I - (A_i \# X^{-1}))$ for $i=1,\dots,m$.
We have
\[
\frac{1}{m} \sum_{i=1}^m \Psi_i(X) = \Phi(X).
\]
There exists \(H_0 > 0\) such that, for each \( i \in \{ 1, \dots, m \} \) and
\( 0 < h \le H_0 \), resolvent maps
\( J_h^{(i)}: \mathbb{T} \to \mathbb{T} \) are defined by
\begin{equation}\label{eq:def-normalized-resolvent-i}
J_h^{(i)}(X) = Y
\quad \Longleftrightarrow \quad
0 = \Psi_i(X) + \frac1h \, (Y-X) + O(h^2),
\end{equation}
in particular \eqref{eq:normalized-BW-expansion} holds.

\begin{theorem}[No-dice theorem]\label{T:Nodice}
For any \(s_{N_0} \in [\alpha I,\beta I]\), where $1/N_0 \leq H_0$, define the cyclic resolvent iteration
\begin{equation}\label{eq:cyclic-iteration-BW}
s_{n+1} = J_{1/(n+1)}^{(i_n)}(s_n),
\qquad
n \ge N_{0},
\end{equation}
where \( i_n \in \{1,\dots,m\} \) is determined by
$i_n \equiv n+1 \pmod m$.
Then
$s_n \to \Omega (\omega, \mathbb{A})$
in the metric $d_\epsilon$.
\end{theorem}

\begin{proof}
Consider a subsequence
$x_n := s_{nm}$ for $nm \ge N_0$.
For \( j=1,\dots,m \),
applying \eqref{eq:def-normalized-resolvent-i} to $s_{nm+j} = J_{1/(nm+j)}^{(j)}(s_{nm+j-1})$ yields
\begin{equation}\label{eq:block-step}
s_{nm+j}
=
s_{nm+j-1} - \frac1{nm+j} \, \Psi_j(s_{nm+j-1}) + r_{n,j},
\qquad
\| r_{n,j} \| \le \frac{C}{(nm+j)^2}.
\end{equation}
From \eqref{eq:block-step} and the bound \( \| \Psi_j \| \le B_0 \), we obtain
\[
\| s_{nm+j}-s_{nm+j-1} \|
\le \frac{B_0}{nm+j} + \frac{C}{(nm+j)^2}
\le \frac{C_1}{n}
\]
for some constant \( C_1 > 0 \), uniformly in \(n \ge 1\) and \( j=1,\dots,m \).
Therefore,
\begin{equation*}\label{eq:intermediate-close}
\| s_{nm+j}-x_n \|
\le
\sum_{\ell=1}^j \| s_{nm+\ell}-s_{nm+\ell-1} \|
\le
\frac{jC_1}{n}
\le \frac{C_2}{n},
\qquad j=0,1,\dots,m,
\end{equation*}
for some constant \( C_2 > 0 \). In particular,
\begin{equation}\label{eq:metric-intermediate-close}
d_\epsilon(s_{nm+j}, x_n) \le \frac{MC_2}{n},
\qquad j=0,1,\dots,m.
\end{equation}
Summing \eqref{eq:block-step} over \(j=1,\dots,m\) yields
\[
x_{n+1}-x_n
=
- \sum_{j=1}^m \frac1{nm+j} \, \Psi_j(s_{nm+j-1})
+ \sum_{j=1}^m r_{n,j}.
\]
We decompose the main sum as follows:
\begin{align*}
\sum_{j=1}^m \frac1{nm+j} \, \Psi_j(s_{nm+j-1})
& =
\frac1{nm}\sum_{j=1}^m \Psi_j(x_n) + \sum_{j=1}^m
\left( \frac1{nm+j} - \frac1{nm} \right) \Psi_j(s_{nm+j-1}) \\
& \hspace{4cm} + \sum_{j=1}^{m} \frac{1}{nm} \left( \Psi_j(s_{nm+j-1}) - \Psi_j(s_{nm}) \right).
\end{align*}
Note that
\begin{displaymath}
\left\Vert \sum_{j=1}^m \left( \frac1{nm+j} - \frac1{nm} \right) \Psi_j(s_{nm+j-1}) \right\Vert
\leq \sum_{j=1}^m \frac{j}{nm(nm+j)} \Vert \Psi_j(s_{nm+j-1}) \Vert
\leq \frac{\hat{B}}{n^{2}}
\end{displaymath}
for some positive constant $\hat{B}$, as well as
\begin{displaymath}
\left\Vert \sum_{j=1}^{m} \frac{1}{nm} \left( \Psi_j(s_{nm+j-1}) - \Psi_j(s_{nm}) \right) \right\Vert
\leq \sum_{j=1}^m \frac{1}{nm} \Vert \Psi_j(s_{nm+j-1}) - \Psi_j(s_{nm}) \Vert
\leq \frac{\hat{C}}{n^{2}}
\end{displaymath}
for some positive constant $\hat{C}$ since $\Psi_{j}$ is bounded and Lipschitz.
So we can write
\begin{displaymath}
\sum_{j=1}^m \frac1{nm+j} \, \Psi_j(s_{nm+j-1}) = \frac1{nm}\sum_{j=1}^m \Psi_j(x_n) + O\!\left(\frac1{n^2}\right)
\end{displaymath}
and then
\begin{equation}\label{eq:block-forward-Euler}
x_{n+1} = x_{n} - \frac1{nm}\sum_{j=1}^m \Psi_j(x_n) + O\!\left(\frac1{n^2}\right)
= x_n - \frac1n \, \Phi(x_n) + O\!\left(\frac1{n^2}\right),
\end{equation}
since $\sum_{j=1}^m \Psi_j = m \Phi$.
Applying Proposition~\ref{P:monotoneResolvent} to the subsequence
\((x_n)\), we obtain
\[
d_\epsilon(\Omega,x_{n+1})
\le
\left( 1 + \frac{a}{n} \right) d_\epsilon(\Omega,x_n) + \frac{C_3}{n^2}
\]
for some constant \( C_3 > 0\). Therefore by Lemma 7.3 in \cite{HallHeyde},
\begin{equation}\label{eq:block-convergence}
d_\epsilon(\Omega,x_n) \longrightarrow 0
\qquad \text{as } n \to \infty.
\end{equation}
By triangle inequality and applying \eqref{eq:metric-intermediate-close} and \eqref{eq:block-convergence}
\begin{displaymath}
d_{\epsilon}(s_{nm+j}, \Omega) \leq d_{\epsilon}(s_{nm+j}, x_{n}) + d_{\epsilon}(x_{n}, \Omega) \longrightarrow 0 \qquad \text{as } n \to \infty,
\end{displaymath}
which means that every subsequence of $s_{n}$ converges to $\Omega$.
Hence, $s_n \to \Omega$ in $d_\epsilon$, which proves the theorem.
\end{proof}

\section{Strong law of large numbers}

Let $0<\alpha<\beta$ be given. The support $\supp{\mu}$ of a Borel probability measure $\mu$ in a Banach (or even a complete metric) space is always a separable set. Furthermore if $\supp(\mu)=1$, that is $\mu$ is fully supported, then $\mu$ is not only $\sigma$-, but in fact is $\tau$-additive and all weak Pettis integrals with respect to $\mu$ restrict to the separable set $\supp(\mu)$, so we may assume that all integrals are actually Bochner integrals with respect to $\mu$ over a Banach space.

Recall that for each $Z\in[\alpha I,\beta I]$ and \(h>0\), the resolvent
\(J_h^{Z}:\mathbb{T}\to \mathbb{T}\) is defined by requiring
\begin{equation}\label{eq:def-normalized-resolvent-Z}
J_h^{Z}(X)=Y
\quad\Longleftrightarrow\quad
0=\Psi_Z(X)+\frac1h\,(Y-X)+O(h^2),
\end{equation}
where $\Psi_Z(X):=(I-(Z\# X^{-1}))$. For instance we may choose $0=\Psi_Z(Y)+\frac1h\,G(Y,X)$ as the defining equation.

Using the same ideas as we did earlier to prove Theorem~\ref{wass_P} and Corollary \ref{uniq}, we arrive at the following:

\begin{theorem}[Wasserstein mean of probability measures with bounded support]\label{wass_P2}
Let $\mu$ be a Borel probability measure with $\supp(\mu)\subseteq([\alpha I,\beta I],\|\cdot\|)$ and $\mu(\supp(\mu))=1$. Then given a $\mathbb{P}$-valued random variable $Y$ with law $\mu$ the equation
\begin{equation}
0=\mathbb{E}_Y\Psi_Y(X)=\int_{\mathbb{P}}\Psi_A(X)d\mu(A)
\end{equation}
has a unique solution $X$ in $\mathbb{P}$ denoted by $\Omega (\mu)$
\end{theorem}

Relying on the No-dice result of the earlier section we establish a stochastic version of it as follows. First, let us recall the following weighted extension of the classical real-valued SLLN due to Etemadi \cite[Theorem 1]{Etemadi}. Consider a sequence of positive real weights $\{w_i\}_i$ such that $W_n = \sum_{i=1}^n w_i \rightarrow \infty$. Assume that
  $$ \sup_n {nw_n \over W_n} < \infty \; \mbox{ and } \; \sup_n \sum_{i=1}^n {i|w_{i+1} - w_i| \over W_n} < \infty.$$
  If $X_1, X_2, \hdots$ is a sequence of real random variables, then
   $$ {1 \over n }\sum_{i=1}^n X_i \rightarrow X_0 \mbox{ a.s. } \quad  \mbox{ implies } \quad {1 \over W_n }\sum_{i=1}^n w_iX_i \rightarrow X_0 \mbox{ a.s.} $$
  Additionally, if the sequence $\{w_i\}_i$
 is monotone, the second condition on the averages above can be omitted.

  \begin{lemma}\label{L:Etemadi}
      Given $0 < \alpha < 1,$ $0 < s_0$ and a sequence $\{b_n\}_n$ of real numbers, let us define the sequence
      $$ s_{n} = \left(1 - {\alpha \over n}\right)s_{n-1} + {\alpha \over n} b_{n-1}, \quad n \geq 1.$$
      Then $s_n = {1 \over W_n} \left(s_0 + \sum_{i=0}^{n-1} w_ib_i \right),$ where the sequence of weights $\{w_i\}_i$ is monotone (increasing or decreasing) and $W_n = \sum_{i=0}^n w_i \rightarrow \infty$ as $n \rightarrow \infty.$
  \end{lemma}

 \begin{proof}
        A straightforward computation gives that
       $$ w_i = {\alpha \over i+1} \prod_{j=1}^{i+1} \left(1 - {\alpha \over j}\right)^{-1} \quad \mbox{ and } \quad W_n = \prod_{j=1}^n \left(1 - {\alpha \over j}\right)^{-1} = O(n^\alpha).$$
       Hence
       $$ {w_{i+1} \over w_i} = { {\alpha \over i+2} \over {\alpha \over i+1}(1- {\alpha \over i+2})} = {
       i+1 \over i+2 - \alpha},$$ which immediately gives the proof of the rest of the lemma.
 \end{proof}

Now we are in position to prove the following stochastic approximation result.

\begin{theorem}[Strong law of large numbers]\label{T:slln}
Let $\mu$ be a Borel probability measure with $\supp(\mu)\subseteq([\alpha I,\beta I],\|\cdot\|)$ and $\mu(\supp(\mu))=1$. Let $\{Y_n\}_{n\in\mathbb{N}}$ be a sequence of $\mathbb{P}$ valued random variables with law $\mu$. For any \(s_{N_0}\in [\alpha I,\beta I]\), where $1/N_0 \leq H_0$ from Theorem~\ref{T:Nodice}, define the stochastic resolvent iteration
\begin{equation}\label{eq:cyclic-iteration-BW2}
s_{n+1}=J_{1/(n+1)}^{Y_{n+1}}(s_n),
\qquad
n\ge N_0,
\end{equation}
Then
$s_n\longrightarrow \Omega (\mu)$ a.s. in the metric $d_\epsilon$.
\end{theorem}

\begin{proof}
Let $a:=\sup\{\Re\sigma(-\mathbb{E}_{Y_1}\Psi_{Y_1}(X)):X\in \mathbb{T}\}+\epsilon$. Let $$L:=\sup_{X,Y,Z\in\mathbb{T}}\frac{\|\Psi_X(Z)-\Psi_Y(Z)\|}{\|X-Y\|}$$ denote the global Lipschitz constant of $\Psi$ on $\mathbb{T}$ with respect to the norm which is finite due to the properties of the functional calculus and differentiability of $\Psi_{(\cdot)}(Z)$ for $Z\in\mathbb{T}$.

Let us define the empirical measure $\mu_k:={1 \over k} \sum_{i=1}^k\delta_{Y_i},$ where $\delta_{Y_i}$ is the Dirac measure supported on $Y_i.$ Then by Varadarajan's theorem \cite[Theorem 11.4.1]{Dudley} the sequence $\mu_k$ a.s. converges weakly to $\mu$ on the support of $\mu$, equivalently we also have $W_1(\mu_k,\mu)\to 0$ a.s. for the $L^1$-Wasserstein distance
\begin{equation*}
W_1(\mu,\nu) := \sup\left\{ \left| \int f(x)d(\mu-\nu)(x) \right|: f \text{ real-valued } L\text{-Lipschitz, } L \leq 1 \right\}
\end{equation*}
induced by the norm on the separable Banach space $\left( \overline{\mathrm{span}\{\supp(\mu),s_0\}}, \| \cdot \| \right)$: see \cite[Theorem 11.3.3]{Dudley}. Then from the dominated convergence theorem for any given $\eta>0$ it follows that
\begin{equation}\label{eq:wass_P2}
\mathbb{E}W_1(\mu_k,\mu) < \eta \quad \text{ for all large enough } k.
\end{equation}

Let us introduce a random sequence of empirical measures $\nu_n:= {1 \over k} \sum_{i=1}^k \delta_{Y_{nk+i}}.$
Then, for any $\hat{s}_0 \in \mathbb{P}$, we can define the random resolvent sequence $$\hat{s}_{n+1}=J_{1/(n+1)}^{\nu_{n+1}}(\hat{s}_{n}) \qquad n \geq N_0.$$
Due to the invariance of $\mathbb{T}$ under the resolvent iterates, both ${s}_{n},\hat{s}_{n}\in\mathbb{T}$ a.s. as well.
From the proof of Proposition~\ref{P:monotoneResolvent} using
\[
u_n'(t)=-\mathbb{E}_{Y_1}\Psi_{Y_1}(u_n(t)),\qquad u_n(0)=\hat{s}_n
\]
with \eqref{eq:normalized-BW-expansion} yields
\begin{equation*}
\hat{s}_{n+1}-u_n(h_n)=h_n \int\Psi_{A}(\hat{s}_n)d(\mu-\nu_{n+1})(A)+O(h_n^2).
\end{equation*}
Furthermore, using
\begin{equation*}
\begin{split}
&\left\|\int\Psi_{A}(\hat{s}_n)d(\mu-\nu_{n+1})(A)\right\|=l_*\left(\int\Psi_{A}(\hat{s}_n)d(\mu-\nu_{n+1})(A)\right)\\
&=\int l_*\left(\Psi_{A}(\hat{s}_n)\right)d(\mu-\nu_{n+1})(A)\leq LW_1(\nu_{n+1},\mu)
\end{split}
\end{equation*}
for a norming linear functional $l_*\in S(\mathcal H)^*$, we obtain the following estimate
\begin{equation}\label{eq1:wass_P2}
\hat{s}_{n+1}-u_n(h_n)= LW_1(\nu_{n+1},\mu)O(h_n)+O(h_n^2),
\end{equation}
and thus,
\[
d_\epsilon(\Omega(\mu),\hat{s}_{n+1})
\le
(1+a h_n) d_\epsilon(\Omega(\mu),\hat{s}_{n}) + L W_1(\nu_{n+1},\mu)O(h_n)+O(h_n^2)
\]
where the implied constants of $O(\cdot)$ are uniform on $\mathbb{T}$.
Since \(h_n=\frac1{n+1}\),
\begin{equation}\label{eq2:wass_P2}
\begin{split}
d_\epsilon(\Omega(\mu),\hat{s}_{n+1})
\le
& \left( 1+\frac{a}{n+1} \right) d_\epsilon(\Omega(\mu),\hat{s}_{n}) + W_1(\nu_{n+1},\mu) O\!\left(\frac{1}{n+1}\right)\\
& + O\!\left(\frac1{(n+1)^2}\right).
\end{split}
\end{equation}
Since $\sum_{n=0}^\infty O\!\left(\frac1{(n+1)^2}\right)<\infty$, the last quadratic error term above is insignificant when we apply Etemadi's weighted extension of the strong law with Lemma~\ref{L:Etemadi} to the stochastic recursion
\begin{equation*}
d(n+1):=\left(1+\frac{a}{n+1}\right)d(n)+W_1(\nu_{n+1},\mu)O\!\left(\frac{1}{n+1}\right)+O\!\left(\frac1{(n+1)^2}\right)
\end{equation*}
agreeing with \eqref{eq2:wass_P2} but '$=$' instead of '$\leq $', thus forcing the a.s. bound $d_\epsilon(\Omega(\mu),\hat{s}_{n})\leq d(n)$. Indeed, by waiting for a suitably large enough starting index $n_0\in\mathbb{N}$ making the drift term $\sum_{n=n_0}^\infty O\!\left(\frac1{(n+1)^2}\right)$ arbitrarily small, and then choosing the starting term as $d(n_0):=d_\epsilon(\Omega(\mu),\hat{s}_{n_0})$, Etemadi SLLN yields $$d(n)\leq O\left(\mathbb{E}W_1(\nu_{1},\mu)\right)+\sum_{n=n_0}^\infty O\!\left(\frac1{(n+1)^2}\right)<\eta$$ a.s. as $n\to\infty$ given that in advance, according to \eqref{eq:wass_P2} as well, large enough $n_0, k$ were chosen. This shows that $d_\epsilon(\Omega(\mu),\hat{s}_{n})\leq \eta$ a.s. for all large enough $n\in\mathbb{N}$.

Now, the proof of the No-dice Theorem~\ref{T:Nodice} leading to \eqref{eq:block-forward-Euler} shows that actually
\begin{equation*}
s_{(n+1)k}=s_{nk}-\frac1n\,\int_{\mathbb{P}}\Psi_A(s_{nk})d\nu_{n+1}(A)+O\!\left(\frac1{n^2}\right)
\end{equation*}
and then the culminating part of the proof of Theorem~\ref{T:Nodice} from \eqref{eq:block-forward-Euler} onward proves that $d_\epsilon(s_{nk},\hat{s}_{n})\longrightarrow 0$ and then again also $d_\epsilon(s_{nk},s_{nk+i})\longrightarrow 0$ for any $1\leq i\leq k$ as well. The proof is concluded.
\end{proof}

\section{Properties associated with the Wasserstein mean}

We see some properties of the Wasserstein mean. Let $\mathbb{A} = (A_{1}, \dots, A_{m}) \in \mathbb{P}^{m}$, any permutation $\sigma$ on $\{ 1, \dots, m \}$, and any $M \in \mathrm{GL}$, the general linear group. We denote as
\begin{displaymath}
\begin{split}
\mathbb{A}_{\sigma} & = (A_{\sigma(1)}, \dots, A_{\sigma(m)}) \in \mathbb{P}^{m}, \\
M \mathbb{A} M^{*} & = (M A_{1} M^{*}, \dots, M A_{m} M^{*}) \in \mathbb{P}^{m}, \\
\mathbb{A}^{k} & = (\underline{A_{1}, \dots, A_{m}}, \dots,
\underline{A_{1}, \dots, A_{m}}) \in \mathbb{P}^{mk},
\end{split}
\end{displaymath}
where the number of blocks in the last expression is $k \in \mathbb{N}$. For given $\omega = (w_{1}, \dots, w_{m}) \in \Delta_{m}$, we also denote as
\begin{displaymath}
\begin{split}
\omega_{\sigma} & = (w_{\sigma(1)}, \dots, w_{\sigma(m)}) \in \Delta_{m}, \\
\omega^{k} & = \frac{1}{k} (\underline{w_{1}, \dots, w_{m}}, \dots,
\underline{w_{1}, \dots, w_{m}}) \in \Delta_{mk}.
\end{split}
\end{displaymath}
Using the uniqueness of a positive definite solution to \eqref{gradF} in Theorem \ref{wass_P} we can simply obtain the following fundamental properties.
\begin{proposition}
Let $\mathbb{A} = (A_{1}, \dots, A_{m}) \in \mathbb{P}^{m}$, and let $\omega = (w_{1}, \dots, w_{m}) \in \Delta_{m}$. Then the following are satisfied.
\begin{itemize}
\item[(1)] $($Consistency with scalars$)$ \ $\displaystyle \Omega(\omega; \mathbb{A}) = \left( \sum_{i=1}^{m} w_{i} A_{i}^{1/2} \right)^{2}$ if $A_{i}$'s commute;

\item[(2)] $($Homogeneity$)$ \ $\Omega(\omega; \alpha \mathbb{A}) = \alpha \Omega(\omega; \mathbb{A})$ for any $\alpha > 0$;

\item[(3)] $($Permutation invariancy$)$ \ $\Omega(\omega_{\sigma}; \mathbb{A}_{\sigma}) = \Omega(\omega; \mathbb{A})$ for any permutation $\sigma$ on $\{ 1, \dots, m \}$;

\item[(4)] $($Repetition invariancy$)$ \ $\Omega(\omega^{k}; \mathbb{A}^{k}) = \Omega(\omega; \mathbb{A})$ for any $k \in \mathbb{N}$;

\item[(5)] $($Unitary congruence invariancy$)$ \ $\Omega(\omega; U \mathbb{A} U^{*}) = U \Omega(\omega; \mathbb{A}) U^{*}$ for any $U \in U(\mathcal{H})$.
\end{itemize}
Furthermore, the equation $\Omega(\omega; A_{1}, \dots, A_{m-1}, X) = X$ has a unique positive definite solution $X = \Omega(\hat{\omega}; A_{1}, \dots, A_{m-1})$ where $\hat{\omega} = \frac{1}{1-w_{m}} (w_{1}, \dots, w_{m-1}) \in \Delta_{m-1}$.
\end{proposition}

\begin{theorem}\label{T:bounded}
Let $\mathbb{A} = (A_1, \ldots, A_m) \in \mathbb{P}^m$ and $\omega = (w_{1}, \dots, w_{m}) \in \Delta_{m}$. The Wasserstein mean satisfies
\begin{displaymath}
2I - \sum_{i=1}^{m} w_{i} A_{i}^{-1} \leq \Omega (\omega, \mathbb{A}) \leq \sum_{i=1}^{m} w_{i} A_{i}.
\end{displaymath}
\end{theorem}

\begin{proof}
Let $X = \Omega (\omega, \mathbb{A})$. Taking squares on both sides of \eqref{gradF} and using the convexity of square map yield
\begin{displaymath}
X^{2} = \left( \sum_{i=1}^{m} w_{i} (X^{1/2} A_{i} X^{1/2})^{1/2} \right)^{2} \leq \sum_{i=1}^{m} w_{i} X^{1/2} A_{i} X^{1/2}.
\end{displaymath}
Taking congruence transformation by $X^{-1/2}$ implies $\displaystyle X \leq \sum_{i=1}^{m} w_{i} A_{i}$.

The equation \eqref{gradF} is equivalent to
\begin{equation} \label{grad-Wass}
I = \sum_{i=1}^{m} w_{i} (X^{-1} \# A_{i}).
\end{equation}
By the geometric-harmonic mean inequality
\begin{displaymath}
I \geq \sum_{i=1}^{m} w_{i} \left( \frac{X + A_{i}^{-1}}{2} \right)^{-1}.
\end{displaymath}
Taking inverse on both sides and applying the arithmetic-harmonic mean inequality, we have
\begin{displaymath}
I \leq \left[ \sum_{i=1}^{m} w_{i} \left( \frac{X + A_{i}^{-1}}{2} \right)^{-1} \right]^{-1} \leq \sum_{i=1}^{m} w_{i} \left( \frac{X + A_{i}^{-1}}{2} \right).
\end{displaymath}
Solving for $X$, we obtain $\displaystyle X \geq 2I - \sum_{i=1}^{m} w_{i} A_{i}^{-1}$.
\end{proof}


\begin{corollary}
Let $\Phi$ be a unital positive linear map. Then
\begin{displaymath}
\Phi(\Omega (\omega, \mathbb{A})) \geq 2I - \sum_{i=1}^{m} w_{i} \Phi(A_{i}^{-1}),
\end{displaymath}
and furthermore,
\begin{displaymath}
\Phi(\Omega (\omega, \mathbb{A})^{-1}) \geq 2I - \sum_{i=1}^{m} w_{i} \Phi(A_{i}).
\end{displaymath}
\end{corollary}

\begin{proof}
The first assertion follows by taking a unital positive linear map $\Phi$ to the first inequality in Theorem \ref{T:bounded}.
For the second assertion, let $X = \Omega (\omega, \mathbb{A})$. By the arithmetic-geometric mean inequality from \eqref{grad-Wass}
\begin{displaymath}
I \leq \sum_{i=1}^{m} w_{i} \left( \frac{X^{-1} + A_{i}}{2} \right).
\end{displaymath}
Then $\displaystyle X^{-1} \geq 2I - \sum_{i=1}^{m} w_{i} A_{i}$ so we obtain the desired inequality by taking a unital positive linear map $\Phi$.
\end{proof}

Following the proof of \cite[Theorem 4.3]{HK19} with Theorem \ref{T:bounded} we obtain
\begin{corollary}
Let $\mathbb{A} = (A_1, \ldots, A_m) \in \mathbb{P}^m$ and $\omega = (w_{1}, \dots, w_{m}) \in \Delta_{m}$. Then
\begin{displaymath}
\lim_{s \to 0} \Omega(\omega; A_{1}^{s}, \dots, A_{m}^{s})^{1/s} = \exp \left( \sum_{i=1}^{m} w_{i} \log A_{i} \right),
\end{displaymath}
whose right-hand side is known as the log-Euclidean mean $\mathrm{LE}(\omega; \mathbb{A})$.
\end{corollary}

The quasi-arithmetic mean of order $p \neq 0$ for $\mathbb{A} = (A_1, \ldots, A_m) \in \mathbb{P}^m$ is defined by
\begin{displaymath}
Q_{p}(\omega; \mathbb{A}) = \left( \sum_{i=1}^{m} w_{i} A_{i}^{p} \right)^{1/p}.
\end{displaymath}
It is a generalized mean including the arithmetic mean $\mathcal{A}$ when $p=1$, log-Euclidean mean as $p \to 0$, and harmonic mean $\mathcal{H}$ when $p=-1$.
Note from \cite{HK25} that the following chain holds: for $1 \leq p \leq q$
\begin{equation} \label{E:chain}
Q_{-q} \leq Q_{-p} \leq \mathcal{H} \preceq Q_{-1/p} \preceq Q_{-1/q} \preceq \mathrm{LE} \preceq Q_{1/q} \preceq Q_{1/p} \leq \mathcal{A} \leq Q_{p} \leq Q_{q},
\end{equation}
where $\preceq$ denotes the near-order: $A \preceq B$ if and only if $A^{-1} \# B \geq I$.

We provide better bounds of Wasserstein mean with respect to an operator norm, rather than the Loewner order in Theorem \ref{T:bounded}.
\begin{theorem}
Let $\mathbb{A} = (A_1, \ldots, A_m) \in \mathbb{P}^m$ and $\omega = (w_{1}, \dots, w_{m}) \in \Delta_{m}$. Then for an operator norm $\Vert \cdot \Vert$
\begin{displaymath}
\left\| \sum_{i=1}^{m} w_{i} A_{i}^{1/2} \right\|^{2} \leq \Vert \Omega(\omega; \mathbb{A}) \Vert \leq \left( \sum_{i=1}^{m} w_{i} \Vert A_{i} \Vert^{1/2} \right)^{2}.
\end{displaymath}
\end{theorem}

\begin{proof}
Note that $\displaystyle \Vert Q_{1/2}(\omega; \mathbb{A}) \Vert = \left\| \sum_{i=1}^{m} w_{i} A_{i}^{1/2} \right\|^{2}$, and the quasi-arithmetic mean and Wasserstein mean are homogeneous. So it is enough for the first inequality to show
\begin{center}
$\Omega(\omega; \mathbb{A}) \leq I$ \quad implies \quad $Q_{1/2}(\omega; \mathbb{A}) \leq I$.
\end{center}
Assume that $X := \Omega(\omega; \mathbb{A}) \leq I$. Then $X^{-1} \geq I$, and by monotonicity of the geometric mean
\begin{displaymath}
\sum_{i=1}^{m} w_{i} A_{i}^{1/2} = \sum_{i=1}^{m} w_{i} (A_{i} \# I) \leq \sum_{i=1}^{m} w_{i} (A_{i} \# X^{-1}) = I.
\end{displaymath}
Hence, $Q_{1/2}(\omega; \mathbb{A}) \leq I$ by taking squares on both sides. Furthermore, by \eqref{gradF}
\begin{displaymath}
\Vert X \Vert = \left\| \sum_{i=1}^{m} w_{i} (X^{1/2} A_{i} X^{1/2})^{1/2} \right\| \leq \sum_{i=1}^{m} w_{i} \Vert X^{1/2} A_{i} X^{1/2} \Vert^{1/2} \leq \sum_{i=1}^{m} w_{i} \Vert X \Vert^{1/2} \Vert A_{i} \Vert^{1/2}.
\end{displaymath}
The first inequality follows from the triangle inequality, and the second follows from the sub-multiplicativity.
Solving for $\Vert X \Vert$, we obtain the desired upper bound for $\Vert X \Vert$.
\end{proof}

\begin{remark}
Note that $A \preceq B$ implies $\Vert A \Vert \leq \Vert B \Vert$ so we obtain from \eqref{E:chain}
\begin{center}
$\Vert \mathrm{LE}(\omega; \mathbb{A}) \Vert \leq \Vert Q_{p}(\omega; \mathbb{A}) \Vert \leq \Vert \Omega(\omega; \mathbb{A}) \Vert$ \quad for all $0 < p \leq 1/2$.
\end{center}
\end{remark}

\begin{theorem}
Let $\mathbb{A} = (A_1, \ldots, A_m) \in \mathbb{P}^m$ and $\omega = (w_{1}, \dots, w_{m}) \in \Delta_{m}$.
\begin{itemize}
  \item[(i)] $\Omega(\omega; \mathbb{A}) \leq I$ implies $\Omega(\omega; \mathbb{A}) \leq Q_{1/2}(\omega; \mathbb{A})^{-1/2}$,
  \item[(ii)] $\Omega(\omega; \mathbb{A}) \geq I$ implies $\Omega(\omega; \mathbb{A}) \geq Q_{1/2}(\omega; \mathbb{A})^{-1/2}$.
\end{itemize}
\end{theorem}

\begin{proof}
Let $X = \Omega(\omega; \mathbb{A})$. By applying Jensen-type inequality to \eqref{gradF}: $(Z^{*} A Z)^{p} \geq Z^{*} A^{p} Z$ for $p \in [0,1]$ and $A \in \mathbb{P}$ when $Z^{-1}$ is a contraction for $Z \in \mathrm{GL}$,
\begin{displaymath}
I \geq X = \sum_{i=1}^{m} w_{i} (X^{1/2} A_{i} X^{1/2})^{1/2} \geq \sum_{i=1}^{m} w_{i} (X^{1/2} A_{i}^{1/2} X^{1/2})
\end{displaymath}
so $\displaystyle X^{-1} \geq \sum_{i=1}^{m} w_{i} A_{i}^{1/2}$. Thus,
\begin{displaymath}
X \leq \left[ \sum_{i=1}^{m} w_{i} A_{i}^{1/2} \right]^{-1} = Q_{1/2}(\omega; \mathbb{A})^{-1/2}.
\end{displaymath}
By a similar argument in (i), we can prove (ii).
\end{proof}

Lim and P\'{a}lfia \cite{LP12} have introduced power means of positive definite matrices but their notion can be extended to the setting $\mathbb{P}$.
For $\mathbb{A} = (A_1, \ldots, A_m) \in \mathbb{P}^m$ and $\omega = (w_{1}, \dots, w_{m}) \in \Delta_{m}$ the power mean $P_{t}(\omega; \mathbb{A})$ with parameter $t \in (0,1]$ is defined by a unique positive definite solution $X$ to the equation
\begin{displaymath}
X = \sum_{i=1}^{m} w_{i} (X \#_{t} A_{i}).
\end{displaymath}
For $t \in [-1,0)$, we define $P_{t}(\omega; \mathbb{A}) := P_{-t}(\omega; \mathbb{A}^{-1})^{-1}$, where $\mathbb{A}^{-1} := (A_1^{-1}, \ldots, A_m^{-1})$.
The remarkable consequence is that power mean $P_{t}$ converges to the Karcher mean $\Lambda$ monotonically as $t \to 0$, in the sense that
\begin{equation*} \label{E:power-means}
P_{-1} = \mathcal{H} \leq P_{-t} \leq P_{-s} \leq \Lambda = P_{0} \leq P_{s} \leq P_{t} \leq \mathcal{A} = P_{1} \quad \textrm{for} \quad 0 < s \leq t \leq 1.
\end{equation*}
We refer \cite{LL14} for more information. Furthermore, the definition of power means can be extended to $t \in (1,2)$: see \cite{Seo}.

\begin{theorem}
Let $\mathbb{A} = (A_1, \ldots, A_m) \in \mathbb{P}^m$ and $\omega = (w_{1}, \dots, w_{m}) \in \Delta_{m}$. Then $\Omega(\omega; \mathbb{A}) \geq I$ implies
\begin{itemize}
  \item[(i)] $\Omega(\omega; \mathbb{A}) \geq P_{t}(\omega; \mathbb{A})^{-1}$ for any $1/2 \leq t <2$, and
  \item[(ii)] $\Omega(\omega; \mathbb{A}) \geq P_{t}(\omega; \mathbb{A}^{-1})$ for any $-2 < t \leq -1/2$.
\end{itemize}
\end{theorem}

\begin{proof}
For $t \in (-2, -1/2]$, $P_{t}(\omega; \mathbb{A}^{-1}) = P_{-t}(\omega; \mathbb{A})^{-1}$ so (i) implies (ii). So it is enough to show (i).

Set $X = \Omega(\omega; \mathbb{A})^{-1}$. We first consider $t \in [1/2,1]$.
By the affine property of parameters for geometric mean and the arithmetic-geometric mean inequality with $\frac{1}{2t} \in [1/2, 1]$, we have
\begin{displaymath}
\begin{split}
I & = \sum_{i=1}^{n} w_{i} (X \# A_{i}) = \sum_{i=1}^{n} w_{i} \left[ X \#_{\frac{1}{2t}} (X \#_{t} A_{i}) \right] \\
& \leq \sum_{i=1}^{n} w_{i} \left[ \left( 1 - \frac{1}{2t} \right) X + \frac{1}{2t} (X \#_{t} A_{i}) \right] \\
& = \left( 1 - \frac{1}{2t} \right) X + \frac{1}{2t} \sum_{i=1}^{n} w_{i} (X \#_{t} A_{i}).
\end{split}
\end{displaymath}
Since $X \leq I$ by assumption, we have
\begin{displaymath}
X \leq 2t I - ( 2t - 1 ) X \leq \sum_{j=1}^{n} w_{j} (X \#_{t} A_{j}) =: g(X).
\end{displaymath}
Since the map $g$ is monotone increasing, we get $X \leq g(X) \leq g^{2}(X) \leq \cdots \leq g^{r}(X)$ for all $r \geq 1$. Note that $g$ is a strict contraction with respect to the Thompson metric, and by the Banach fixed point theorem $g^{r}(X)$ converges to a unique fixed point as $r \to \infty$, which is the power mean $P_{1/2}(\omega; \mathbb{A})$.
Therefore,
\begin{center}
$\Omega(\omega; \mathbb{A}) = X^{-1} \geq P_{1/2}(\omega; \mathbb{A})^{-1} \geq P_{t}(\omega; \mathbb{A})^{-1}$ \quad for all $t \in [1/2,1]$.
\end{center}
Since $P_{t=1}(\omega; \mathbb{A}) = \mathcal{A}(\omega; \mathbb{A}) \leq P_{t}(\omega; \mathbb{A})$ for $t \in (1,2)$
from \cite[Proposition 4.5 (v)]{Seo}, we obtain the conclusion.
\end{proof}


Hwang and Kim \cite{HK19} have defined the Wasserstein mean of $A, B \in \mathbb{P}$ for $t \in [0,1]$ as a unique positive definite solution to the equation
\begin{equation} \label{E:Wass-eq}
I = (1-t) (A \# X^{-1}) + t (B \# X^{-1}).
\end{equation}
Indeed, the unique solution $X \in \mathbb{P}$ is given by
\begin{displaymath}
A \diamond_{t} B = (1-t)^{2} A + t^{2} B + t(1-t) \left[ A (A^{-1} \# B) + (A^{-1} \# B) A \right].
\end{displaymath}
This can be written as
\begin{equation} \label{E:Wass-expression}
\begin{split}
A \diamond_{t} B & = A^{-1/2} [ (1-t)A + t(A^{1/2} B A^{1/2})^{1/2} ]^{2} A^{-1/2} \\
& = \left[ I \nabla_{t} (A^{-1} \# B) \right] A \left[ I \nabla_{t} (A^{-1} \# B)
\right],
\end{split}
\end{equation}
where $A \nabla_{t} B := (1-t) A + t B$ is the weighted arithmetic mean of $A$ and $B$.

\begin{theorem}
For $A, B \in \mathbb{P}$ and $t \in [0,1]$
\begin{displaymath}
A \diamond_{t} B = \min \left\{
X \in \mathbb{P}: (1-t)
\begin{bmatrix}
  A^{-1} & I \\
  I & X
\end{bmatrix} + t
\begin{bmatrix}
  A^{-1} & A^{-1} \# B \\
  A^{-1}  \# B & X
\end{bmatrix} \geq 0
\right\}.
\end{displaymath}
\end{theorem}

\begin{proof}
Note from \cite[Theorem 1.3.3]{Bh} that
\begin{center}
$\displaystyle \begin{bmatrix} A & X \\ X^{*} & B \end{bmatrix} \geq 0$ \quad if and only if \quad $X B^{-1} X^{*} \leq A$.
\end{center}
Assume that
\begin{displaymath}
\begin{bmatrix}
  A^{-1} & I \nabla_{t} (A^{-1} \# B) \\
  I \nabla_{t} (A^{-1} \# B) & X
\end{bmatrix}
= (1-t)
\begin{bmatrix}
  A^{-1} & I \\
  I & X
\end{bmatrix} + t
\begin{bmatrix}
  A^{-1} & A^{-1} \# B \\
  A^{-1} \# B & X
\end{bmatrix} \geq 0
\end{displaymath}
Then $[ I \nabla_{t} (A^{-1} \# B) ] X^{-1} [ I \nabla_{t} (A^{-1} \# B) ] \leq A^{-1}$ so
\begin{displaymath}
X \geq [ I \nabla_{t} (A^{-1} \# B) ] A [ I \nabla_{t} (A^{-1} \# B) ] = A \diamond_{t} B.
\end{displaymath}
If $X = A \diamond_{t} B$, then $[ I \nabla_{t} (A^{-1} \# B) ] X^{-1} [ I \nabla_{t} (A^{-1} \# B) ] = A^{-1}$ by the Riccati equation. So
\begin{displaymath}
(1-t)
\begin{bmatrix}
  A^{-1} & I \\
  I & A \diamond_{t} B
\end{bmatrix} + t
\begin{bmatrix}
  A^{-1} & A^{-1} \# B \\
  A^{-1} \# B & A \diamond_{t} B
\end{bmatrix} \geq 0.
\end{displaymath}
\end{proof}

The Wasserstein nean $A \diamond_{t} B$ can be defined for all $t \in \mathbb{R}$ by \eqref{E:Wass-expression}, and $A \diamond_{t} B \in \overline{\mathbb{P}}$ since $I \nabla_{t} (A^{-1} \# B) \in S(\mathcal{H})$. Moreover, $A \diamond_{t} B \in \mathbb{P}$ for some $t \in \mathbb{R}$ when $I \nabla_{t} (A^{-1} \# B)$ is invertible.

\begin{lemma} \label{L:Wass-symmetry}
Let $A, B \in \mathbb{P}$ and $t \in \mathbb{R}$. Then the Wasserstein mean satisfies the symmetry:
\begin{displaymath}
A \diamond_t B =  B \diamond_{1-t} A.
\end{displaymath}
\end{lemma}

\begin{proof}
The proof is the same as that of \cite[Lemma 2.6]{HK19}: the main idea is to use \eqref{E:Wass-expression} and to apply the symmetry, self-duality and Riccati equation of geometric mean so the proof of \cite[Lemma 2.6]{HK19} holds for $t \in \mathbb{R}$.
\end{proof}

\begin{theorem}
Let $A, B \in \mathbb{P}$. Then $A \diamond_{t} B$ for some $t \in \mathbb{R}$ such that $I \nabla_{t} (A^{-1} \# B) \in \mathbb{P}$ is the unique solution $X \in \mathbb{P}$ to the equation \eqref{E:Wass-eq}.
\end{theorem}

\begin{proof}
It is known from \cite[Theorem 5.1]{HK19} that $A \diamond_{t} B$ is the unique solution $X \in \mathbb{P}$ to the equation \eqref{E:Wass-eq} for $t \in [0,1]$.

Let $t > 1$. Taking the congruence transformation by $A^{-1/2}$ in \eqref{E:Wass-eq}, we get
\begin{displaymath}
A^{-1} = (1-t)(A^{-1/2} X^{-1} A^{-1/2})^{1/2} + t(A^{-1/2} B A^{-1/2}) \# (A^{-1/2} X^{-1} A^{-1/2}).
\end{displaymath}
Set $Y := A^{-1/2} X^{-1} A^{-1/2}$ and $Z := A^{-1/2} B A^{-1/2}$. Then $A^{-1} = (1-t)Y^{1/2} + t(Z \# Y)$. Since $A^{-1} - (1-t) Y^{1/2} \in \mathbb{P}$, by the Riccati equation we have
\begin{displaymath}
[A^{-1} - (1-t)Y^{1/2}] Y^{-1} [A^{-1} - (1-t)Y^{1/2}] = t^{2} Z.
\end{displaymath}
Pre-multiplying and post-multiplying all terms by $A$ yield $\left[ Y^{-1/2} - (1-t) A \right]^{2} = t^{2} AZA$, so $Y^{-1/2} = (1-t)A + t(AZA)^{1/2}$. By assumption, we have
\begin{displaymath}
(A^{1/2}XA^{1/2})^{1/2}= (1-t)A + t(A^{1/2}BA^{1/2})^{1/2}.
\end{displaymath}
Taking square on both sides and applying the congruence transformation by $A^{-1/2}$, we obtain $X = A^{-1/2} [ (1-t)A + t(A^{1/2}BA^{1/2})^{1/2} ]^{2} A^{-1/2} = A \diamond_{t} B$.

Let $t < 0$. By substituting $s = 1-t > 1$, we get from the preceding argument that the equation \eqref{E:Wass-eq} is equivalent to
 \begin{displaymath}
I = s (A \# X^{-1}) + (1-s) (B \# X^{-1}),
\end{displaymath}
and has a unique solution $X = B \diamond_{s} A$. By Lemma \ref{L:Wass-symmetry} $X = B \diamond_{1-t} A = A \diamond_{t} B$.
\end{proof}

The Karcher mean $\Lambda$ satisfies the self-duality, meanwhile in general for the Wasserstein mean it does not hold: $\Omega(\omega; \mathbb{A}) \neq \Omega(\omega; \mathbb{A}^{-1})^{-1}$.
So we call $\Omega(\omega; \mathbb{A}^{-1})^{-1}$ a dual mean of the Wasserstein mean.
It is the unique positive definite solution to the equation
\begin{displaymath}
I = \mathcal{H}(\omega; A_{1} \sharp X^{-1}, \dots, A_{m} \sharp X^{-1}),
\end{displaymath}
where $\mathcal{H}$ denotes the harmonic mean. This is equivalent to
\begin{equation} \label{E:dual-1}
\sum_{j=1}^{m} w_{j} (X \sharp A_{j}^{-1}) = I.
\end{equation}
Applying congruence transformation by $X^{-1/2}$ into \eqref{E:dual-1} and taking inverses, we can see that $\Omega(\omega; \mathbb{A}^{-1})^{-1}$ is the unique positive definite solution of
\begin{displaymath}
X = \left[ \sum_{j=1}^{m} w_{j} (X^{1/2} A_{j} X^{1/2})^{-1/2} \right]^{-1}.
\end{displaymath}

We show the relationships between the dual of the Wasserstein mean and Karcher mean, and the quasi-arithmetic mean.
\begin{theorem}
Let $\mathbb{A} = (A_{1}, \dots, A_{m}) \in \mathbb{P}^{m}$ and $\omega = (w_{1}, \dots, w_{m}) \in \Delta_{m}$. Then
\begin{center}
$\Omega(\omega; \mathbb{A}^{-1})^{-1} \geq I$ \quad implies \quad $\Lambda(\omega; \mathbb{A}^{1/2})^{2} \geq I$.
\end{center}
\end{theorem}

\begin{proof}
Assume that $X = \Omega(\omega; \mathbb{A}^{-1})^{-1} \geq I$. Then
\begin{displaymath}
X = \left[ \sum_{j=1}^{m} w_{j} (X^{1/2} A_{j} X^{1/2})^{-1/2} \right]^{-1}.
\end{displaymath}
By the Karcher-harmonic mean inequality,
\begin{displaymath}
X \leq \Lambda (\omega; (X^{1/2} A_{1} X^{1/2})^{1/2}, \dots, (X^{1/2} A_{m} X^{1/2})^{1/2}).
\end{displaymath}
Since $X \geq I$, $(X A X)^{p} \leq X A^{p} X$ for any $A \in \mathbb{P}$ and $0 \leq p \leq 1$ by Jensen-type inequality. So the preceding inequality yields
\begin{displaymath}
X \leq \Lambda (\omega; X^{1/2} A_{1}^{1/2} X^{1/2}, \dots, X^{1/2} A_{m}^{1/2} X^{1/2}),
\end{displaymath}
due to the monotonicity of Karcher mean. Taking congruence transformation by $X^{-1/2}$, we obtain $I \leq \Lambda(\omega; \mathbb{A}^{1/2})$, which completes the proof.
\end{proof}

\begin{theorem} \label{T:Dual-Q}
Let $\mathbb{A} = (A_{1}, \dots, A_{m}) \in \mathbb{P}^{m}$ and $\omega = (w_{1}, \dots, w_{m}) \in \Delta_{m}$. Then
\begin{center}
$\Omega(\omega; \mathbb{A}^{-1})^{-1} \leq I$ \quad implies \quad $Q_{-1/2}(\omega; \mathbb{A}) \leq I$.
\end{center}
\end{theorem}

\begin{proof}
Assume that $X = \Omega(\omega; \mathbb{A}^{-1})^{-1} \leq I$.
Since $X \leq I$, $(X^{1/2} A X^{1/2})^{p} \geq X^{1/2} A^{p} X^{1/2}$ for any $A \in \mathbb{P}_{m}$ and $0 \leq p \leq 1$ by Jensen-type inequality.
So $(X^{1/2} A X^{1/2})^{-1/2} \leq X^{-1/2} A^{-1/2} X^{-1/2}$, and
\begin{displaymath}
\begin{split}
X & = \left[ \sum_{j=1}^{m} w_{j} (X^{1/2} A_{j} X^{1/2})^{-1/2} \right]^{-1} \\
& \geq \left[ \sum_{j=1}^{m} w_{j} (X^{-1/2} A_{j}^{-1/2} X^{-1/2}) \right]^{-1} = X^{1/2} \left[ \sum_{j=1}^{m} w_{j} A_{j}^{-1/2} \right]^{-1} X^{1/2}.
\end{split}
\end{displaymath}
Taking congruence transformation by $X^{-1/2}$ yields $\displaystyle \left[ \sum_{j=1}^{m} w_{j} A_{j}^{-1/2} \right]^{-1} \leq I$. It completes the proof by taking squares on both sides.
\end{proof}

Since the dual mean $\Omega(\omega; \mathbb{A}^{-1})^{-1}$ and quasi-arithmetic mean $Q_{-1/2}(\omega; \mathbb{A})$ are both homogeneous, Theorem \ref{T:Dual-Q} implies from \cite[Lemma 2.2]{JK23}
\begin{corollary}
Let $\mathbb{A} = (A_{1}, \dots, A_{m}) \in \mathbb{P}^{m}$ and $\omega = (w_{1}, \dots, w_{m}) \in \Delta_{m}$. Then
\begin{displaymath}
\Vert Q_{-1/2}(\omega; \mathbb{A}) \Vert \leq \Vert \Omega(\omega; \mathbb{A}^{-1})^{-1} \Vert.
\end{displaymath}
\end{corollary}

\begin{proposition} \label{P:Wass-dual}
Let $A, B \in \mathbb{P}$ and $t \in [0,1]$. Then
\begin{displaymath}
(A^{-1} \diamond_{t} B^{-1})^{-1} \preceq A \diamond_{t} B.
\end{displaymath}
\end{proposition}

\begin{proof}
Note from \cite{HK} that
\begin{displaymath}
A \diamond_{t} B = \left[ I \nabla_{2t(1-t)} (X \nabla X^{-1}) \right] (A^{-1} \diamond_{t} B^{-1})^{-1} \left[ I \nabla_{2t(1-t)} (X \nabla X^{-1}) \right],
\end{displaymath}
where $X = A^{-1} \# B$. By the Riccati equation
\begin{displaymath}
(A^{-1} \diamond_{t} B^{-1}) \# (A \diamond_{t} B) = I \nabla_{2t(1-t)} (X \nabla X^{-1}) \geq I.
\end{displaymath}
Indeed, the inequality holds since $X \nabla X^{-1} \geq I$ and $2t(1-t) \in [0,1/2]$. By definition of near-order, we obtain the conclusion.
\end{proof}

From Proposition \ref{P:Wass-dual} one can naturally ask whether the following holds:
\begin{displaymath}
\Omega(\omega; \mathbb{A}^{-1})^{-1} \preceq \Omega(\omega; \mathbb{A}).
\end{displaymath}
Using the following lemma about a reverse inequality of arithmetic-geometric mean inequality,
we show the relationship between the Wasserstein mean and its dual with respect to the near order.

\begin{lemma} \label{L:converse}
Let $A, B \in \mathbb{P}$ such that $m_{1} I \leq A \leq M_{1} I$ and $m_{2} I \leq B \leq M_{2} I$ for some positive constants $m_{1}, m_{2}, M_{1}$ and $M_{2}$. Then
\begin{displaymath}
A \nabla B \leq \kappa (A \# B),
\end{displaymath}
where $\displaystyle \kappa := \max \left\{ K \left( \frac{M_{1}}{m_{2}} \right), K \left( \frac{M_{2}}{m_{1}} \right) \right\}$ for the Kantorovich constant $\displaystyle K(h) = \frac{1 + h}{2 \sqrt{h}}$.
\end{lemma}

\begin{proof}
Set $X := A^{-1/2} B A^{-1/2}$. Then the inequality which we need to show is equivalent to
\begin{displaymath}
\frac{I + X}{2} \leq \kappa X^{1/2}.
\end{displaymath}
Let $\displaystyle f(t) = \frac{1+t}{2\sqrt{t}}$ for $t > 0$. Then $\displaystyle f'(t) = \frac{t-1}{4 t^{3/2}}$ so it is increasing on $(1,\infty)$ and decreasing on $(0,1)$. Assuming that $t \in [\alpha, \beta]$ we obtain
\begin{equation} \label{E:conv-ineq}
\frac{1+t}{2} \leq c \sqrt{t}, \quad c = \max \{ f(\alpha), f(\beta) \}.
\end{equation}
Since $\displaystyle \frac{m_{2}}{M_{1}} I \leq X \leq \frac{M_{2}}{m_{1}} I$, \eqref{E:conv-ineq} implies $I \nabla X \leq \kappa X^{1/2}$, where
\begin{displaymath}
\kappa = \max \left\{ f \left( \frac{m_{2}}{M_{1}} \right), f \left( \frac{M_{2}}{m_{1}} \right) \right\} = \max \left\{ K \left( \frac{M_{1}}{m_{2}} \right), K \left( \frac{M_{2}}{m_{1}} \right) \right\},
\end{displaymath}
since $K(h) = K(h^{-1})$ for $h > 0$.
\end{proof}

\begin{theorem}
Let $\omega = (w_{1}, \dots, w_{m}) \in \Delta_{m}$ and $\mathbb{A} = (A_1, \ldots, A_m) \in \mathbb{P}^m$ such that $m I \leq A_{i} \leq M I$ for some constants $0 < m \leq M$. Then
\begin{displaymath}
\kappa^{-4} \Omega(\omega; \mathbb{A}) \preceq \Omega(\omega; \mathbb{A}^{-1})^{-1} \preceq (2 \kappa - 1)^{2} \Omega(\omega; \mathbb{A}),
\end{displaymath}
where $\displaystyle \kappa = \max \left\{ \frac{M^{2} + 1}{2 M}, \frac{m^{2} + 1}{2 m} \right\}$.
\end{theorem}

\begin{proof}
Let $X = \Omega(\omega; \mathbb{A}^{-1})^{-1}$ and $Y = \Omega(\omega; \mathbb{A})$.
By \eqref{grad-Wass} and joint concavity of the geometric mean
\begin{equation} \label{E:ineq-1}
I = \sum_{i=1}^{m} w_{i} \left[ \frac{(X \# A_{i}^{-1}) + (Y^{-1} \# A_{i})}{2} \right]
\leq \sum_{i=1}^{m} w_{i} \left( \frac{X + Y^{-1}}{2} \right) \# \left( \frac{A_{i}^{-1} + A_{i}}{2} \right).
\end{equation}
Assuming that $m I \leq A_{i} \leq M I$ for some constants $0 < m \leq M$, we obtain $m I \leq X, Y \leq M I$. By Lemma \ref{L:converse}
\begin{center}
$X \nabla Y^{-1} \leq \kappa (X \# Y^{-1})$ \quad and \quad $A_{i}^{-1} \nabla A_{i} \leq \kappa (A_{i}^{-1} \# A_{i}) = \kappa I$,
\end{center}
where $\displaystyle \kappa = \max \left\{ K(M^{2}), K(m^{2}) \right\}$.
By monotonicity of the geometric mean, \eqref{E:ineq-1} implies
\begin{displaymath}
I \leq \sum_{i=1}^{m} \kappa w_{i} (X \# Y^{-1}) \# I = \kappa (X \# Y^{-1})^{1/2}.
\end{displaymath}
This is equivalent to $I \leq \kappa^{2} (X \# Y^{-1}) = X \# (\kappa^{-4} Y)^{-1}$, which gives us the first inequality.

By \eqref{grad-Wass} and joint concavity of the geometric mean
\begin{equation} \label{E:ineq-2}
I = \left[ \sum_{i=1}^{m} w_{i} (X \# A_{i}^{-1}) \right] \# \left[ \sum_{i=1}^{m} w_{i} (Y^{-1} \# A_{i}) \right]
\geq \sum_{i=1}^{m} w_{i} \left[ (X \# A_{i}^{-1}) \# (Y^{-1} \# A_{i}) \right].
\end{equation}
By Lemma \ref{L:converse}
\begin{center}
$X \# A_{i}^{-1} \geq \kappa^{-1} (X \nabla A_{i}^{-1})$ \quad and \quad $Y^{-1} \# A_{i} \geq \kappa^{-1} (Y^{-1} \nabla A_{i})$,
\end{center}
and by monotonicity of the geometric mean, \eqref{E:ineq-2} implies
\begin{displaymath}
\begin{split}
I & \geq \sum_{i=1}^{m} \kappa^{-1} w_{i} \left[ (X \nabla A_{i}^{-1}) \# (Y^{-1} \nabla A_{i}) \right] \\
& \geq \kappa^{-1} \sum_{i=1}^{m} w_{i} \left( \frac{X \# Y^{-1} + A_{i}^{-1} \# A_{i}}{2} \right) = \frac{\kappa^{-1} (X \# Y^{-1}) + \kappa^{-1} I}{2}.
\end{split}
\end{displaymath}
By simplification, $\displaystyle X \# Y^{-1} \leq (2 \kappa - 1)I$, at which we obtain the second assertion.
\end{proof}

\section*{Acknowledgments}
M.~P\'alfia and V. N.~Mer would like to thank Zolt\'an L\'eka for stimulating discussions about semigroups in Banach spaces.

The work of S. Kim was supported by the National Research Foundation of Korea grant funded by the Korea government (MSIT) (No. NRF-2022R1A2C4001306). The work of V. N. Mer  was supported by Basic Science Research Program through the National Research Foundation of Korea (NRF) funded by the Ministry of Education, Korea (No. RS-2024-00462498). The work of M. P\'alfia was supported by the Ministry of Innovation and Technology of Hungary from the National Research, Development and Innovation Fund and financed under the TKP2021-NVA funding scheme, Project no. TKP2021-NVA-09; the Hungarian Scientific Research fund NKFIH ADVANCED-150059 and by the János Bolyai Research Scholarship of the Hungarian Academy of Sciences, Grant No. BO/00998/23/3.

%

\end{document}